\documentclass[11pt,reqno]{amsart}
\usepackage{amsmath,a4wide,scalerel}
\usepackage{amsaddr}
\usepackage{mathabx}
\usepackage{xcolor,graphicx}
\usepackage{mathrsfs}
\usepackage{pgfplots}
\usepackage{subcaption}
\usepackage[normalem]{ulem}
\usepackage{float}
\usepackage{accents}
\usepackage[pagewise]{lineno}
\pgfplotsset{compat=1.18}
\usepackage{bbm}

\usepackage{hyperref}
\newtheorem{theorem}{Theorem}
\newtheorem{lemma}[theorem]{Lemma}
\newtheorem{proposition}[theorem]{Proposition}
\newtheorem{remark}[theorem]{Remark}
\newtheorem{corollary}[theorem]{Corollary}
\newtheorem{assumption}{Assumption}
\newtheorem{example}{Example}

\newcommand{\N}{{\mathbb N}}
\newcommand{\E}{{\mathbb E}}
\newcommand{\PP}{{\mathbb P}}

\newcommand*\diff{\mathop{}\!\mathrm{d}}

\title[]{Boundary-preserving Lamperti-splitting schemes for some stochastic differential equations}

\author[]{Johan Ulander}
\address{Department of Mathematical Sciences,
Chalmers University of Technology, 41296~Gothenburg, Sweden}
\email{\tt johanul@chalmers.se}



\begin{document}

\begin{abstract}
We propose and analyse boundary-preserving schemes for the strong approximations of some scalar SDEs with non-globally Lipschitz drift and diffusion coefficients whose state-space is bounded. The schemes consists of a Lamperti transform followed by a Lie--Trotter splitting. We prove $L^{p}(\Omega)$-convergence of order $1$, for every $p \geq 1$, of the schemes and exploit the Lamperti transform to confine the numerical approximations to the state-space of the considered SDE. We provide numerical experiments that confirm the theoretical results and compare the proposed Lamperti-splitting schemes to other numerical schemes for SDEs.
\end{abstract}

\maketitle
\begin{sloppypar}

\section{Introduction}\label{sec:intro}
Stochastic differential equations (SDEs) are nowadays widely used to model various phenomena. Classical examples are found in physics, engineering, financial mathematics, mathematical biology, epidemic modelling, etc. \cite{GraySIS, Karlin1981ASC, introTostocCalc, MR1214374, MR2001996}. An important feature of some SDEs is that the state-space is a strict subset of the target-space. In this article we propose numerical schemes whose approximations only take values in the state-space of the considered SDE. We say that a numerical scheme with this property is boundary-preserving. Boundary-preserving methods have received a lot of attention the last two decades. Without being exhaustive, we mention the following articles \cite{MR3006996, pospre, positivity, MR4220738,MR3433041, BPAit, MR4544037, DomPres, MR2341800, MR3248050, Schurz1995NumericalRF, MR4274899, Yang2021FirstOS, MR4268206}. Applications include, for example, population dynamics models \cite{MR4220738, MR4274899, Yang2021FirstOS, MR4268206} and heat flow models \cite{pospre,positivity}, defined e.g. in $[0,1]$ and defined in $(0,\infty)$, respectively.

In this work we propose and study order $1$ strongly convergent splitting schemes for scalar time-homogeneous Itô-type SDEs
\begin{equation}\label{eq:SDE}
\left\lbrace
\begin{aligned}
& \diff X(t) = f(X(t)) \diff t + g(X(t)) \diff B(t),\ t \in (0,T], \\ 
& X(0) = x_{0} \in D,
\end{aligned}
\right.
\end{equation}
where $T>0$, $f,g : \mathbb{R} \to \mathbb{R}$ are given functions satisfying some conditions to be specified in Section~\ref{sec:setting}, $B$ is a standard Brownian motion and $D \subset \mathbb{R}$ is the bounded state-space of the solution to the SDE in \eqref{eq:SDE}. We assume that $x_{0}$ is deterministic and given. Typically, $f$ and $g$ are non-globally Lipschitz functions and hence the SDE in equation~\eqref{eq:SDE} is not covered by classical SDE theory. As we will see in Section~\ref{sec:setting}, under some assumptions and provided that $x_{0} \in \mathring{D} = \{ x \in D:\ x \not\in \partial D \}$, where $\partial D$ denotes the boundary of $D$, the solution $X$ takes values in the interior $\mathring{D}$ of the domain $D$. For precise definition of the setting, see Section~\ref{sec:setting}. 

Examples of applications include some instances of the Susceptible–Infected–Susceptible (SIS) epidemic model \cite{MR4220738, MR3015083, GraySIS, MR4274899, Yang2021FirstOS, MR4268206}, the Nagumo SDE \cite{MR3308418, MCKEAN1970209} and an Allen--Cahn type SDE \cite{ALLEN19791085, MR3986273, Funaki1995TheSL, MR3308418}. We consider the SIS epidemic model corresponding to the choices $f(x) = x - x^2$ and $g(x) = x - x^2$ and is also known as the simplest Wright--Fisher diffusion for a gene frequency model. The Nagumo SDE corresponds to the choices $f(x) = -x (1-x) (1-x)$ and $g(x) = - x + x^2$. The Allen--Cahn type SDE corresponds to the choices $f(x) = x - x^3$ and $g(x) = 1 - x^2$. We provide short discussions and motivations for these models in Section~\ref{sec:num}.

The proposed schemes combine the Lamperti transform with a time splitting procedure. The Lamperti transform applied to the SDE in \eqref{eq:SDE} guarantees that the schemes are boundary-preserving. We employ a Lie--Trotter time splitting of the resulting transformed SDE to obtain tractable sub-problems. The proposed numerical schemes are defined in equation~\eqref{eq:LampSplitDisc} and in equation~\eqref{eq:LampSplitDisc2}.

The main results of the paper are the following:
\begin{itemize}
\item We propose approximation procedures for SDEs of the form in~\eqref{eq:SDE} that is boundary-preserving, see Proposition~\ref{propo:BP-scheme} and Proposition~\ref{propo:BP-scheme_v2}.
\item We prove $L^{p}(\Omega)$-convergence of order $1$ for every $p \geq 1$, see Theorem~\ref{theo:main_v2} and Theorem~\ref{theo:main_v3}, and almost sure pathwise convergence of order $1-\epsilon$, for every $\epsilon>0$, see Corollary~\ref{cor:as_conv} and Corollary~\ref{cor:as_conv2}.
\end{itemize}

The literature on schemes based on the Lamperti transform and on time splitting schemes is extensive. Without being exhaustive, we mention the following articles \cite{MR4220738, BPAit, MR4544037, MR3248050, Yang2021FirstOS, MR4268206} on schemes based on the Lamperti transform and the following references \cite{pospre,positivity, MR3986273, MR2840298, MR3433041, BPAit, MR4544037, MR2009376, MR2341800} on time splitting schemes for differential equations. To the best of our knowledge, only the two recent articles \cite{BPAit, MR4544037} combines these two approaches to construct a positivity-preserving scheme for the Aït-Sahalia model and the Cox–Ingersoll–Ross (CIR) process, respectively. The CIR model considered in \cite{MR4544037} has an affine function as drift coefficient and the diffusion coefficient is $1/2$-Hölder continuous. In the present paper, we consider drift and diffusion coefficients that can have superlinear growth.

Before closing the introduction, we would like to compare the proposed schemes to the literature on numerical schemes based on the Lamperti transform on similar problems. We first mention the paper \cite{MR4220738}, where the authors prove strong convergence of order $1$ for a family of stochastic SIS equations using a Lamperti transform followed by smoothing the drift coefficient. The smoothing strategy in \cite{MR4220738} enables the authors to obtain $L^{2}(\Omega)$-convergence of order $1$ for quite general drifts (essentially requiring $C^{2}$ on the closure of the domain) and for a diffusion coefficient of the form $x(1-x)$, exploiting for example inverse moment bounds of the exact solution and exponential integrability of the transformed SDE. In this work we consider similar drifts coefficients (requiring $C^{2}$ on the closure of the domain, see Assumption~\ref{ass:f}, and a decay condition near the boundary points, see Assumption~\ref{ass:fg}) and more general diffusion coefficients (See Assumption~\ref{ass:g}). After using the Lamperti transform, we apply a Lie--Trotter splitting. This approach enables us to establish representation formulas (see equation~\eqref{eq:expressYLS} and equation~\eqref{eq:expresshatYLS}) for the approximate solutions that are similar to the integral equation for the exact solution of the SDE. From this we obtain $L^{p}(\Omega)$-convergence of order $1$, for every $p \geq 1$, and almost sure pathwise convergence of order $1-\epsilon$ for every $\epsilon>0$. We also mention the articles \cite{MR4274899, Yang2021FirstOS, MR4268206} and \cite{MR3006996, MR3248050}, where the authors apply the Lamperti transform followed by the (truncated) Euler–Maruyama (EM) schemes and the semi-implicit Euler–Maruyama (SEM) scheme, respectively, to the transformed SDEs. \cite{MR4274899} considers SIS SDEs and the authors obtain improved, although not as general, results compared to \cite{MR4220738} discussed above. \cite{MR4268206} establishes $L^{p}(\Omega)$-convergence and almost sure pathwise convergence for Lamperti (truncated) EM schemes for general SDEs defined in $(0,\infty)$. \cite{MR3006996, MR3248050} obtain $L^{p}(\Omega)$-convergence of order $1$ for the Lamperti SEM scheme for some SDEs defined in domains under slightly different conditions on the drift and diffusion coefficients. By the assumptions in Sections~\ref{sec:setting} and~\ref{sec:scheme_v2}, such explicit and implicit Lamperti-based schemes are also covered in the proposed schemes in the present setting. In future works, however, the proposed approach could possibly be extended to cases where Lamperti (truncated) EM and Lamperti SEM are not applicable.

This paper is organized as follows. Section~\ref{sec:setting} is devoted to presenting the setting, assumptions and some properties of the considered SDE. In Section~\ref{sec:scheme} and Section~\ref{sec:scheme_v2} we define the semi-analytic and fully-discrete, respectively, Lamperti-splitting schemes and state and prove boundary-preservation and strong convergence.  Lastly, in order to support our theoretical results in Section~\ref{sec:scheme} and Section~\ref{sec:scheme_v2}, we provide numerical experiments in Section~\ref{sec:num}. 

\section{Setting}\label{sec:setting}
In this section we introduce the notation and the assumptions on the considered SDE~\eqref{eq:SDE}. Let $(\Omega,\mathcal{F},\mathbb{P})$ be a fixed probability space equipped with a filtration $\bigl(\mathcal{F}_t\bigr)_{t\ge 0}$ that satisfies the usual conditions. We denote by $\mathbb{E}[\cdot]$ the expectation operator and $C(a_{1},\ldots,a_{l})$ denotes a (non-random) generic constant that depends on the parameters $a_{1},\ldots,a_{l}$ and that may vary from line to line. Throughout the paper, most equalities and inequalities should be understood in the almost sure sense, we do not specify this everywhere to avoid repetition.

\subsection{Description of the SDE}\label{sec:sde}
We first discuss some preliminaries and introduce the main assumptions needed for the definition and analysis of the proposed Lamperti-splitting (LS) schemes. The general idea of the Lamperti transform is to transform an SDE into another SDE with state-independent diffusion coefficient \cite{MR1214374,  Mller2010FromSD}. More precisely, provided that everything is well-defined, the Lamperti transform of the SDE in equation~\eqref{eq:SDE} with $x \in \mathring{D}$ is given by
\begin{equation}\label{eq:LampTrans}
\Phi(x) = \int_{w_{0}}^{x} \frac{1}{g(w)} \diff w
\end{equation}
where $\mathring{D} = (l,r)$ for some $l, r \in \mathbb{R}$ and for some $w_{0} \in \mathring{D}$. We also let $\bar{D} = [l,r]$. By Itô's formula, the process $Y(t) = \Phi(X(t))$ satisfies
\begin{equation}\label{eq:SDE-Lamp}
\left\lbrace
\begin{aligned}
& \diff Y(t) = \tilde{H}(Y(t)) \diff t + \diff B(t),\\ 
& Y(0) = \Phi(x_{0}),
\end{aligned}
\right.
\end{equation}
where we define
\begin{equation*}
\tilde{H}(x) = \frac{f(\Phi^{-1}(x))}{g(\Phi^{-1}(x))} - \frac{1}{2} g'(\Phi^{-1}(x)).
\end{equation*}
We require $x_{0}, w_{0} \in \mathring{D}$ because, by Assumption~\ref{ass:g} below, $\Phi$ has singularities on $\partial D$. Also observe that $w_{0} = x_{0}$ is a valid choice for the lower integration limit in the Lamperti transform. Let us also denote by
\begin{equation}\label{eq:Hdef}
H(x) = \tilde{H}(x) - \mu
\end{equation}
for some $\mu \in \mathbb{R}$. The introduction of $H$ allows us to transfer the constant $\mu$ between the ODE part and the SDE part of the splitting scheme (see Section~\ref{sec:schemeconstruction}). We now list the assumptions that we need to guarantee that the above is well-defined.
\begin{assumption}\label{ass:f}
The drift coefficient $f \in C^{2} \left(\bar{D}\right)$.
\end{assumption}
\begin{assumption}\label{ass:g}
The diffusion coefficient $g \in C^{3} \left(\bar{D} \right) $ and is strictly positive on $\mathring{D} = (l,r)$,  and, for any $w_{0} \in \mathring{D}$, the following non-integrability conditions are satisfied
\begin{equation}\label{eq:intCond}
\int_{w_{0}}^{l} \frac{1}{g(w)} \diff w = - \infty,\ \int_{w_{0}}^{r} \frac{1}{g(w)} \diff w = \infty.
\end{equation}
\end{assumption}
\begin{assumption}\label{ass:fg}
The drift coefficient $f$ decays at least as fast as the diffusion coefficient $g$ near the boundary points $\partial D$; that is, the following limits exist and are finite
\begin{equation}\label{eq:limCond}
\left| \lim_{x \searrow l} \frac{f(x)}{g(x)} \right| + \left| \lim_{x \nearrow r} \frac{f(x)}{g(x)} \right| < \infty.
\end{equation}
\end{assumption}
The above assumptions implies regularity properties of the inverse of the Lamperti transform $\Phi^{-1}$ and of the modified drift $H$ in equation~\eqref{eq:Hdef} that we summarise in the following two propositions.
\begin{proposition}\label{prop:phiinv}
Suppose Assumption~\ref{ass:g} is satisfied. Then $\Phi^{-1}: \mathbb{R} \to \mathring{D}$ is bounded, bijective, continuously differentiable and has bounded derivative. In particular, $\Phi^{-1}: \mathbb{R} \to \mathring{D}$ is globally Lipschitz continuous and we denote the Lipschitz constant of $\Phi^{-1}$ by $L_{\Phi^{-1}}$. 
\end{proposition}
\begin{proof}[Proof of Proposition~\ref{prop:phiinv}]
The assumption that $g(x)>0$ for every $x \in \mathring{D}$ and the non-integrability conditions in equation~\eqref{eq:intCond} implies that $\Phi^{-1}: \mathbb{R} \to \mathring{D}$ is bijective and differentiable. The latter follows from the inverse function theorem since $\Phi'(x) = \frac{1}{g \left( \Phi^{-1}(x) \right)} > 0$ for every $x \in \mathring{D}$. For boundedness of $ \frac{\diff}{\diff x} \Phi^{-1}(x)$, we can compute, using the inverse function theorem and the chain rule,
\begin{equation*}
\frac{\diff}{\diff x} \Phi^{-1}(x) = \frac{1}{\Phi'\left( \Phi^{-1}(x) \right)} = g \left( \Phi^{-1}(x) \right),
\end{equation*}
for every $x \in \mathbb{R}$. Since $g$ is continuous on $\bar{D}$, by Assumption~\ref{ass:g}, and $\Phi^{-1}(x) \in \mathring{D}$ for every $x \in \mathbb{R}$, the above expressions implies that $\frac{\diff}{\diff x} \Phi^{-1}(x)$ is continuous and uniformly bounded. Since any differentiable function that has a uniformly bounded derivative is globally Lipschitz continuous, we conclude that $\Phi^{-1}$ is globally Lipschitz continuous. 
\end{proof}
Observe that we could continue the argument in the proof of Proposition~\ref{prop:phiinv} to conclude that $\Phi^{-1} \in C^{4} \left( \mathbb{R} \right)$ with bounded derivatives up to order $4$.
\begin{proposition}\label{prop:H}
Suppose Assumptions~\ref{ass:f},~\ref{ass:g} and~\ref{ass:fg} are satisfied. Then $H \in C^{2} \left(\mathbb{R}\right)$ and
\begin{equation*}
L_{H} = \sup_{x \in \mathbb{R}} |H(x)+\mu| + \sup_{x \in \mathbb{R}} |H'(x)| + \sup_{x \in \mathbb{R}} |H''(x)| < \infty
\end{equation*}
and is independent of $\mu$. In particular, $H, H': \mathbb{R} \to \mathbb{R}$ are globally Lipschitz continuous and $L_{H}$ is an upper bound for the Lipschitz constants of $H$ and $H'$.
\end{proposition}
Observe that we include
\begin{equation*}
\sup_{x \in \mathbb{R}} | H(x) + \mu | = \sup_{x \in \mathbb{R}} |\tilde{H}(x)|
\end{equation*}
in the definition of $L_{H}$ in order for it to be independent of $\mu$ (but dependent on $f$ and $g$). Also note that Proposition~\ref{prop:H} implies that
\begin{equation*}
\sup_{x \in \mathbb{R}} |H(x)| \leq L_{H} + |\mu|.
\end{equation*}
\begin{proof}[Proof of Proposition~\ref{prop:H}]
Recall that
\begin{equation}\label{eq:H}
H(x) = \frac{f(\Phi^{-1}(x))}{g(\Phi^{-1}(x))} - \frac{1}{2} g'(\Phi^{-1}(x)) - \mu.
\end{equation}
By the quotient rule, the chain rule and by the inverse function theorem, we can compute
\begin{equation}\label{eq:Hprim}
H'(x) = f' \left( \Phi^{-1}(x) \right) - \frac{f \left( \Phi^{-1}(x) \right) g' \left( \Phi^{-1}(x) \right)}{g \left( \Phi^{-1}(x) \right)} - \frac{1}{2} g'' \left( \Phi^{-1}(x) \right) g \left( \Phi^{-1}(x) \right),
\end{equation}
for every $x \in \mathbb{R}$, and
\begin{equation}\label{eq:Hprimprim}
  \begin{split}
    H''(x) &= f'' \left( \Phi^{-1}(x) \right) g \left( \Phi^{-1}(x) \right) - f' \left( \Phi^{-1}(x) \right) g' \left( \Phi^{-1}(x) \right) \\
    &- f \left( \Phi^{-1}(x) \right) g'' \left( \Phi^{-1}(x) \right) + \frac{f \left( \Phi^{-1}(x) \right) \left( g' \left( \Phi^{-1}(x) \right) \right)^{2}}{g \left( \Phi^{-1}(x) \right)} \\
    &- \frac{1}{2} g''' \left( \Phi^{-1}(x) \right) \left( g \left( \Phi^{-1}(x) \right) \right)^{2}  \\
    &- \frac{1}{2} g'' \left( \Phi^{-1}(x) \right) g' \left( \Phi^{-1}(x) \right) g \left( \Phi^{-1}(x) \right),
  \end{split}
\end{equation}
for every $x \in \mathbb{R}$. By Assumptions~\ref{ass:f} and~\ref{ass:g} all terms in equations~\eqref{eq:H},~\eqref{eq:Hprim} and~\eqref{eq:Hprimprim} are continuous and bounded, except possibly the terms containing division by $g \left( \Phi^{-1}(x) \right)$. Since $g(x) > 0$ for every $x \in \mathring{D}$ by Assumption~\ref{ass:g} and $\Phi^{-1}(x) \in \mathring{D}$ for every $x \in \mathbb{R}$ by Proposition~\ref{prop:phiinv}, it suffices to verify that the quotients do not blow up as $x \to -\infty$ and $x \to \infty$. Observe that this is equivalent to that the corresponding quotients without $\Phi^{-1}(x)$ do not blow up as $x \searrow l$ and as $x \nearrow r$. But this follows from Assumption~\ref{ass:fg} and from uniform boundedness of $g' \left( \Phi^{-1}(x) \right)$:
\begin{equation*}
\left| \lim_{x \searrow l} \frac{f(x) g'(x)}{g(x)} \right| \leq \sup_{y \in \bar{D}} |g'(y)| \left| \lim_{x \searrow l} \frac{f(x)}{g(x)} \right| < \infty
\end{equation*}
and
\begin{equation*}
\left| \lim_{x \searrow l} \frac{f(x) \left( g'(x) \right)^{2}}{g(x)} \right| < \sup_{y \in \bar{D}} |g'(y)|^{2} \left| \lim_{x \searrow l} \frac{f(x)}{g(x)} \right| < \infty
\end{equation*}
and similarly for $x \nearrow r$. Thus, $H+\mu$, $H'$ and $H''$ are continuous and uniformly bounded by the constant $L_{H}$ that is independent of $\mu$.
\end{proof}
We remark that it is essential for any numerical scheme that utilises the Lamperti transform that $\Phi$ is well-defined, which Assumption~\ref{ass:g} guarantees by Proposition~\ref{prop:phiinv}. Moreover, that $\Phi: \mathring{D} \to \mathbb{R}$ is bijective and continuous implies, using also Assumptions~\ref{ass:f} and~\ref{ass:fg} by Proposition~\ref{prop:H}, that $\partial D$ is unattainable by $X$ without reference to Feller's boundary classification (see Section~\ref{sec:BB}). For a detailed and elaborate discussion on Feller's boundary classification see, for example, \cite{Karlin1981ASC}.

Assumption~\ref{ass:g} is satisfied in, for example, \cite{MR4220738, MR4274899, Yang2021FirstOS} where the authors obtain $L^{2}(\Omega)$-convergence of order $1$. On the other hand, Assumption~\ref{ass:g} is, for example, not satisfied for the CIR model and the authors of \cite{MR4544037} do not recover $L^{2}(\Omega)$-convergence of order $1$. Observe that for the CIR model, the ODE in~\eqref{eq:ODE-Lamp} admits an exact solution and so the methodology in this work could be applicable. But a different proof strategy has to be used, as neither Assumption~\ref{ass:g} nor Assumption~\ref{ass:fg} are satisfied for the CIR model.

A key step in the construction of the proposed LS schemes is to apply a Lie--Trotter time splitting to the SDE in~\eqref{eq:SDE-Lamp}: we iteratively solve the nonlinear ODE
\begin{equation}\label{eq:ODE-Lamp}
\frac{\diff y(t)}{\diff t} = H(y(t))
\end{equation}
and the SDE for Brownian motion with drift
\begin{equation*}
\diff Z(t) = \mu \diff t + \diff B(t).
\end{equation*}
Proposition~\ref{prop:H} guarantees a unique and global solution to the ODE in equation~\eqref{eq:ODE-Lamp}.

At this point, we would like to mention an explicit and interesting class of possible choices of $f$ and $g$ that satisfy Assumptions~\ref{ass:f},~\ref{ass:g} and~\ref{ass:fg}. Let $g$ be a polynomial with no roots in $\mathring{D}$ and zeros at $x=l$ and $x=r$; that is, $g$ can be represented as
\begin{equation*}
g(x) = (x - l)^{\alpha_{l}} (x - r)^{\alpha_{r}} \tilde{g}(x)
\end{equation*}
where $\alpha_{l}, \alpha_{r} \in \{ 1,2,3,\ldots \}$ are the multiplicities of the roots $x=l$ and $x=r$, respectively, of $g$ and where $\tilde{g}$ is some polynomial with no roots in $\bar{D}$. Then $g$ satisfies Assumption~\ref{ass:g}. Let $f$ be another polynomial. Then $f$ satisfies Assumption~\ref{ass:f}. In order for Assumption~\ref{ass:fg} to be satisfied, $f$ must also have zeros at $x=l$ and $x=r$; that is, $f$ can be represented as
\begin{equation*}
f(x) = (x-l)^{\beta_{l}} (x-r)^{\beta_{r}} \tilde{f}(x)
\end{equation*}
where $\beta_{l}, \beta_{r} \in \{ 1,2,3,\ldots \}$ are the multiplicities of the roots $x=l$ and $x=r$, respectively, of $f$ and where $\tilde{f}$ is some polynomial with no roots at $x=l$ and $x=r$. Moreover, by L'Hôpital's rule, the multiplicities of the zeros $x=l$ and $x=r$ of $f$ must be at least as high as the multiplicities of the corresponding zeros of $g$; that is, $\alpha_{l} \leq \beta_{l}$ and $\alpha_{r} \leq \beta_{r}$. Then $f$ and $g$ satisfy Assumptions~\ref{ass:f},~\ref{ass:g} and~\ref{ass:fg} and are thus covered in the framework of this paper. In fact, the numerical examples that we provide in Section~\ref{sec:num} all belong to this class of polynomials.

After this preparation, we can define the class of SDEs that we consider in this work. We consider scalar time-homogeneous stochastic differential equations in the Itô sense
\begin{equation}\label{eq:SDEmain}
\left\lbrace
\begin{aligned}
& \diff X(t) = f(X(t)) \diff t + g(X(t)) \diff B(t),\ t \in (0,T], \\ 
& X(0) = x_{0} \in \mathring{D},
\end{aligned}
\right.
\end{equation}
where $T>0$ and $f$ and $g$ are such that Assumption~\ref{ass:f},~\ref{ass:g}, and~\ref{ass:fg} and are satisfied. We say that a stochastic process $\left( X(t) \right)_{t \in [0,T]}$ is a (strong) solution of \eqref{eq:SDEmain} if the corresponding integral equation
\begin{equation*}
X(t) = x_{0} + \int_{0}^{t} f(X(s)) \diff s + \int_{0}^{t} g(X(s)) \diff B(s)
\end{equation*}
is satisfied, almost surely, for every $t \in [0,T]$, where the second integral is an Itô integral. Naturally, the above definition requires that the involved integrals are well-defined. We refer the interested reader to \cite{MR2001996} for details on well-posedness of (strong) solutions of SDEs with Lipschitz continuous coefficients. The well-posedness of (strong) solutions to \eqref{eq:SDEmain} follows from the well-posedness of (strong) solutions $Y$ to the SDE in~\eqref{eq:SDE-Lamp} with drift coefficient of class $C^{2}$ with bounded derivatives and additive diffusion coefficient, $X = \Phi^{-1}(Y)$ (see equations~\eqref{eq:LampTrans} and~\eqref{eq:SDE-Lamp}), and the above assumptions.

\subsection{Boundary classification}\label{sec:BB}
We dedicate this section to a short discussion about whether or not the process $X$ can hit the boundary points $\partial D$, where the Lamperti transform $\Phi$ is not well-defined. The boundary points are unattainable by $X$ if and only if the stopping time
\begin{equation*}
\tau = \inf \{ t \in (0,T]:\ X(t) \in \partial D \}
\end{equation*}
is infinite almost surely. In the considered setting, we have the following characterisation of when $\tau = \infty$ almost surely:
\begin{equation*}
\PP( \tau = \infty) = 1 \Longleftrightarrow \PP(Y \text{ blows up}) = 0.
\end{equation*}
By Proposition~\ref{prop:H}, the drift coefficient $H$ of the transformed process $Y$ is Lipschitz continuous. Thus, $Y$ does not blow up in finite time with probability $1$, and we conclude that $\tau = \infty$ almost surely. Alternatively, Feller's boundary classification provides a general theory on boundary behaviour of solutions to Itô SDEs and characterises this in terms of the drift and diffusion coefficients $f$ and $g$. For a detailed exposition of Feller's boundary classification we refer the interested reader to \cite{Karlin1981ASC}.

\section{A semi-analytic boundary-preserving integrator}\label{sec:scheme}
In this section we present and state the properties of the semi-analytic boundary-preserving integrator for the SDE in~\eqref{eq:SDEmain}. We say semi-analytic in this section because we assume that the nonlinear ODE in equation~\eqref{eq:ODE-Lamp} admits an exact solution, ideally with an analytical formula. See also equation~\eqref{eq:ODE-LampMain} below. In Section~\ref{sec:scheme_v2} we will instead use a numerical method to solve the ODE in~\eqref{eq:ODE-Lamp} and we will there obtain a boundary-preserving integrator for the SDE in~\eqref{eq:SDEmain}.

We partition the interval $[0,T]$ into $M \in \N$ subintervals $[t_m,t_{m+1}]$, each of length $\Delta t=T/M$. This means that $t_{m} = m \Delta t$, for $m=0, \ldots, M$.

We propose a scheme based on utilising the Lamperti transform associated with the considered SDE in equation~\eqref{eq:SDE-Lamp} followed by a Lie--Trotter splitting strategy of the resulting transformed SDE in equation~\eqref{eq:SDE-Lamp}.

We first provide a detailed description of the construction of the scheme in Section~\ref{sec:schemeconstruction}. We then provide the main results of this section in Section~\ref{sec:results}; that is, the boundary-preserving property of the scheme (Proposition~\ref{propo:BP-scheme}) and the $L^{p}(\Omega)$-convergence of order $1$ for every $p \geq 1$ (Theorem~\ref{theo:main_v2}). As a corollary, we also obtain almost sure pathwise convergence of order $1-\epsilon$ for every $\epsilon>0$ (Corollary~\ref{cor:as_conv}). 

\subsection{Description of the integrator}\label{sec:schemeconstruction}
In the following we describe how the semi-analytic Lamperti-splitting (LS) scheme is constructed. We consider the SDE obtained by the transformation $Y(t) = \Phi(X(t))$, i.e. equation~\eqref{eq:SDE-Lamp},
\begin{equation}\label{eq:SDE-LampMain}
\left\lbrace
\begin{aligned}
& \diff Y(t) = \left( H(Y(t)) + \mu \right) \diff t + \diff B(t),\\ 
& Y(0) = \Phi(x_{0}),
\end{aligned}
\right.
\end{equation}
using Itô's lemma, where $H$ and $\mu$ are defined in Section~\ref{sec:setting}. We construct an approximation $Y^{LS}$ of the solution $Y$ to the SDE in equation~\eqref{eq:SDE-LampMain} and define the approximation $X^{LS}$ of the solution $X$ to the original SDE in equation~\eqref{eq:SDEmain} as $X^{LS} = \Phi^{-1}(Y^{LS})$. We construct the semi-analytic LS scheme on the time grid points $0 = t_{0} < \ldots < t_{M} = T$, denoted by $X^{LS}_{m}$ and $Y^{LS}_{m}$ for $m=0, \ldots, M$. We let $Y^{LS}_{0} = \Phi(x_{0})$ and we define $Y^{LS}_{m}$, for $m=1, \ldots, M$, recursively as follows: Suppose $Y^{LS}_{m}$ at time $t_m=m\Delta t$ is given. First we let $y_{m}$ solve the nonlinear ODE
\begin{equation}\label{eq:ODE-LampMain}
\left\lbrace
\begin{aligned}
& \frac{\diff y_{m}(t)}{\diff t} = H(y_{m}(t)),\\ 
& y_{m}(t_{m}) = Y^{LS}_{m},
\end{aligned}
\right.
\end{equation}
on the interval $[t_{m},t_{m+1}]$ with initial value $Y^{LS}_{m}$. Second we let $Z_{m}$ solve the SDE for Brownian motion with drift $\mu \in \mathbb{R}$
\begin{equation}\label{eq:SDE-main}
\left\lbrace
\begin{aligned}
& \diff Z_{m}(t) = \mu \diff t + \diff B(t),\\ 
& Z_{m}(t_{m}) = y_{m}(t_{m+1}),
\end{aligned}
\right.
\end{equation}
on the interval $[t_{m},t_{m+1}]$ with initial value $y_{m}(t_{m+1})$.  We define $Y^{LS}_{m+1}$ at the next time grid point $t_{m+1} = (m+1) \Delta t$ as
\begin{equation}\label{eq:LampSplitDisc}
Y^{LS}_{m+1} = Z_{m}(t_{m+1}) \equiv y_{m}(t_{m+1}) + \mu \Delta t + B(t_{m+1}) - B(t_{m})
\end{equation}
and we define the approximation $X^{LS}_{m+1} = \Phi^{-1}(Y^{LS}_{m+1})$ of the solution of equation~\eqref{eq:SDEmain} at the time grid point $t_{m+1}$.

\begin{remark}
The same scheme could have been constructed by considering the opposite order of the splitting; that is, first solve the SDE in equation~\eqref{eq:SDE-main} and then solve the ODE in equation~\eqref{eq:ODE-LampMain}. Then we would directly obtain $L^{2}(\Omega)$-convergence of order $1$ by applying Theorem $3.1$ in \cite{MR2341800}. However, Theorem $3.1$ in \cite{MR2341800} does not yield the stronger results of $L^{p}(\Omega)$-convergence rate, for every $p \geq 1$, with the supremum inside the expected value: see Theorem~\ref{theo:main_v2} below.
\end{remark}
For the proof of the convergence result in this section (see Theorem~\ref{theo:main_v2}), the following integral expression for $Y^{LS}_{m+1}$, for $m=0, \ldots, M-1$, will be used
\begin{equation}\label{eq:LampSplitInt}
Y^{LS}_{m+1} = Y^{LS}_{m} + \int_{t_{m}}^{t_{m+1}} H(y_{m}(s)) \diff s + \mu \Delta t + B(t_{m+1}) - B(t_{m}).
\end{equation}

The definition of $X^{LS}$ as $\Phi^{-1}(Y^{LS})$ combined with Assumptions~\ref{ass:f},~\ref{ass:g} and~\ref{ass:fg} guarantee that the scheme is boundary-preserving. This is the content of the next proposition.

\begin{proposition}\label{propo:BP-scheme}
Let $M\in\N$, $T > 0$, $\Delta t = T/M$ and let $x_{0} \in \mathring{D}$.
Suppose that Assumptions~\ref{ass:f},~\ref{ass:g} and~\ref{ass:fg} are satisfied.
Let $Y^{LS}$ be given by the splitting scheme in equation~\eqref{eq:LampSplitDisc} and let $X^{ LS} = \Phi^{-1}(Y^{LS})$ be the numerical approximation of the original SDE in~\eqref{eq:SDEmain}. Then
\begin{equation*}
X^{LS}_{m} \in \mathring{D},
\end{equation*}
almost surely, for every $m \in \{0,\ldots, M\}$.
\end{proposition}
\begin{proof}[Proof of Proposition~\ref{propo:BP-scheme}]
Recall that $\Phi^{-1}: \mathbb{R} \to \mathring{D}$ is bijective and continuous, by Proposition~\ref{prop:phiinv}. Thus, if $Y^{LS}_{m}$, for $m = 0,\ldots, M$, does not blow up, with probability $1$, then the statement holds. Suppose that $Y^{LS}_{m}$ is finite with probability $1$ for some $m = 0, \ldots M-1$. Then $y_{m}(t_{m+1})$, in equation~\eqref{eq:ODE-LampMain}, is finite with probability $1$ as it is the solution of the globally well-defined ODE in~\eqref{eq:ODE-LampMain} at time $t_{m+1}$. Since $Z_{m}(t)$ in equation~\eqref{eq:SDE-main}, for $t \in [t_{m},t_{m+1}]$, is a Brownian motion with drift $\mu$ and unit diffusion starting at $y_{m+1}(t_{m+1})$, we have that $Y^{LS}_{m+1} = Z_{m}(t_{m+1})$ is finite with probability $1$. This gives the desired result.
\end{proof}

\subsection{Convergence result}\label{sec:results}
In the following we prove that the proposed semi-analytic LS scheme has $L^{p}(\Omega)$-convergence of order $1$, for every $p \geq 1$, for the considered SDE in equation~\eqref{eq:SDEmain}.

We would like to make a few remarks and comparisons to other related papers before we state and prove the convergence result.
\begin{remark}
The positive moments of the exact solution and of the LS approximation are immediately bounded as both are confined to the bounded domain $\mathring{D}$. 
\end{remark}
\begin{remark}
In contrast to the papers \cite{MR4220738, MR4274899, Yang2021FirstOS, MR4268206}, we do not need to bound the inverse moments of the exact solution and of the LS approximation to obtain our convergence result.
\end{remark}

\begin{theorem}\label{theo:main_v2}
Let $M\in\N$, $T > 0$, $\Delta t = T/M$ and let $x_{0} \in \mathring{D}$.
Suppose Assumptions~\ref{ass:f},~\ref{ass:g} and~\ref{ass:fg} are satisfied. Let $X^{LS} = \Phi^{-1}(Y^{LS})$, where $Y^{LS}$ is defined by the splitting scheme in equation~\eqref{eq:LampSplitDisc}, and let $X$ be the exact solution of the considered SDE in equation~\eqref{eq:SDEmain}. Then, for every $p \geq 1$, it holds
\begin{equation*}
\left( \E \left[ \sup_{m=0, \ldots, M} | X^{LS}_{m} - X(t_{m}) |^{p} \right] \right)^{\frac{1}{p}} \leq C(p,\Phi^{-1},H,T,\mu) \Delta t.
\end{equation*}
\end{theorem}
The dependence of the constant $C(p,\Phi^{-1},H,T,\mu)$ in Theorem~\Ref{theo:main_v2} on the parameters can be found in the proof.
\begin{proof}[Proof of Theorem~\ref{theo:main_v2}]
As $X^{LS}_{0} = X(0)$, it suffices to consider the case $m \in \{ 1, \ldots, M \}$. Since $\Phi^{-1}$ is globally Lipschitz continuous by Proposition~\ref{prop:phiinv}, the proof reduces to proving the corresponding $L^{p}(\Omega)$-estimate for $|Y^{LS}_{m}-Y(t_{m})|$:
\begin{equation*}
| X^{LS}_{m} - X(t_{m}) | = | \Phi^{-1}(Y^{LS}_{m}) - \Phi^{-1}(Y(t_{m})) | \leq L_{\Phi^{-1}} | Y^{LS}_{m} - Y(t_{m}) |.
\end{equation*}
Recall the following integral expression for $Y^{LS}_{m}$ stated in equation~\eqref{eq:LampSplitInt} in Section~\ref{sec:schemeconstruction}
\begin{equation*}
  \begin{split}
    Y^{LS}_{m} &= Y^{LS}_{m-1} + \int_{t_{m-1}}^{t_{m}} H(y_{m-1}(s)) \diff s + \mu \Delta t + B(t_{m}) - B(t_{m-1}) \\
    &= Y^{LS}_{m-1} + \int_{t_{m-1}}^{t_{m}} \left( H(y_{m-1}(s)) + \mu \right) \diff s + B(t_{m}) - B(t_{m-1}).
  \end{split}
\end{equation*}
By recursively applying the above formula, we obtain that
\begin{equation}\label{eq:expressYLS}
Y^{LS}_{m} = Y(0) + \sum_{k=0}^{m-1} \int_{t_{k}}^{t_{k+1}} \left( H(y_{k}(s)) + \mu \right) \diff s + B(t_{m})
\end{equation}
and we re-write the integral representation of $Y$ on a similar form
\begin{equation*}
  \begin{split}
    Y(t_{m}) &= Y(0) + \int^{t_{m}}_{0} \left( H(Y(s)) + \mu) \right) \diff s + B(t_{m}) \\
    &= Y(0) + \sum_{k=0}^{m-1} \int_{t_{k}}^{t_{k+1}} \left( H(Y(s)) + \mu \right) \diff s + B(t_{m})
  \end{split}
\end{equation*}
at $t=t_{m}$. Therefore
\begin{equation}\label{eq:LSerr1}
|Y^{LS}_{m}-Y(t_{m})| \leq \left| \sum_{k=0}^{m-1} \int_{t_{k}}^{t_{k+1}} H \left(y_{k}(s)\right) - H \left(Y(s) \right) \diff s \right|,
\end{equation}
We first split the integral as follows
\begin{equation}\label{eq:err1}
\begin{split}
\int_{t_{k}}^{t_{k+1}} H \left(y_{k}(s) \right) - H \left(Y(s) \right) \diff s &= \int_{t_{k}}^{t_{k+1}} H \left(y_{k}(s) \right) - H \left(Y^{LS}_{k} \right) \diff s \\ &+ \int_{t_{k}}^{t_{k+1}} H \left( Y^{LS}_{k} \right) - H \left( Y(t_{k})\right) \diff s \\ &+ \int_{t_{k}}^{t_{k+1}} H \left( Y(t_{k}) \right) - H \left( Y(s) \right) \diff s
\end{split}
\end{equation}
and consider each term separately. Observe that the second term on the right hand side of equation~\eqref{eq:err1} is the term that we need for Grönwall's lemma after the estimate and simplification
\begin{equation*}
\begin{split}
\left| \int_{t_{k}}^{t_{k+1}} H \left( Y^{LS}_{k} \right) - H \left( Y(t_{k}) \right) \diff s \right| &\leq L_{H} \int_{t_{k}}^{t_{k+1}} |Y^{LS}_{k} - Y(t_{k})| \diff s \\ &= L_{H} \Delta t |Y^{LS}_{k} - Y(t_{k})|,
\end{split}
\end{equation*}
where we used that $H$ is globally Lipschitz continuous with Lipschitz constant bounded by $L_{H}$ by Proposition~\ref{prop:H}. Recall now that $y_{k}(s) = Y^{LS}_{k} + \int_{t_{k}}^{s} H(y_{k}(r)) \diff r$ is the exact solution to the ODE in equation~\eqref{eq:ODE-LampMain} starting at $t_{k}$ with initial value $Y^{LS}_{k}$. Therefore, the first term on the right hand side in equation~\eqref{eq:err1} can be estimated by
\begin{equation*}
\begin{split}
\left| \int_{t_{k}}^{t_{k+1}} H \left( y_{k}(s) \right) -H \left( Y^{LS}_{k} \right) \diff s \right| &\leq L_{H} \int_{t_{k}}^{t_{k+1}} | y_{k}(s)-Y^{LS}_{k}| \diff s \\ &\leq L_{H} \int_{t_{k}}^{t_{k+1}} \int_{t_{k}}^{s} |H(y_{k}(r))| \diff r \diff s \\ &\leq L_{H} \left( L_{H} + |\mu| \right) \Delta t^{2},
\end{split}
\end{equation*}
since $H$ is globally Lipschitz continuous with Lipschitz constant bounded by $L_{H}$ and uniformly bounded by $L_{H}+|\mu|$ by Proposition~\ref{prop:H}. It remains to consider the third term on the right hand side in equation~\eqref{eq:err1}
\begin{equation*}
\int_{t_{k}}^{t_{k+1}} H \left( Y(t_{k}) \right) - H \left( Y(s) \right) \diff s.
\end{equation*}
By Proposition~\ref{prop:H}, $H \in C^{2} \left( \mathbb{R} \right)$ and we can thus apply Itô's formula
\begin{equation*}
\begin{split}
H(Y(s)) = H(Y(t_{k})) + \int_{t_{k}}^{s} H(Y(r)) H'(Y(r)) \diff r &+ \frac{1}{2} \int_{t_{k}}^{s} H''(Y(r))  \diff r \\ &+ \int_{t_{k}}^{s} H'(Y(r)) \diff B(r)
\end{split}
\end{equation*}
to see that
\begin{equation}\label{eq:Hito1}
\begin{split}
\int_{t_{k}}^{t_{k+1}} H \left( Y(t_{k}) \right) - H \left( Y(s) \right) \diff s  &= \int_{t_{k}}^{t_{k+1}} \int_{t_{k}}^{s} H(Y(r)) H'(Y(r)) \diff r \diff s \\ &+ \frac{1}{2} \int_{t_{k+1}}^{t_{k+1}} \int_{t_{k}}^{s} H''(Y(r)) \diff r \diff s \\ &+ \int_{t_{k}}^{t_{k+1}} \int_{t_{k}}^{s} H'(Y(r)) \diff B(r) \diff s.
\end{split}
\end{equation}
To obtain convergence order $1$, we have to sum over $k=0,\ldots,m-1$ before estimating the separate terms in equation~\eqref{eq:Hito1}
\begin{equation}\label{eq:Hito21}
\begin{split}
\sum_{k=0}^{m-1} \int_{t_{k}}^{t_{k+1}} H \left( Y(t_{k}) \right) - H \left( Y(s) \right) \diff s  &= \int_{0}^{t_{m}} \int_{\ell(s)}^{s} H(Y(r)) H'(Y(r)) \diff r \diff s \\ &+ \frac{1}{2} \int_{0}^{t_{m}} \int_{\ell(s)}^{s} H''(Y(r)) \diff r \diff s \\ &+ \int_{0}^{t_{m}} \int_{\ell(s)}^{s} H'(Y(r)) \diff B(r) \diff s,
\end{split}
\end{equation}
where we let $\ell(s) = t_{k}$ whenever $s \in [t_{k},t_{k+1})$. The first two terms on the right hand side of equation~\eqref{eq:Hito21} can estimated by
\begin{equation*}
\begin{split}
\left| \int_{0}^{t_{m}} \int_{\ell(s)}^{s} H(Y(r)) H'(Y(r)) \diff r \diff s \right| &\leq \int_{0}^{t_{m}} \int_{\ell(s)}^{s} | H(Y(r)) H'(Y(r))| \diff r \diff s \\ &\leq \left( L_{H} + |\mu| \right) L_{H} t_{m} \Delta t
\end{split}
\end{equation*}
and
\begin{equation*}
\left| \frac{1}{2} \int_{0}^{t_{m}} \int_{\ell(s)}^{s} H''(Y(r)) \diff r \diff s \right| \leq \frac{1}{2} \int_{0}^{t_{m}} \int_{\ell(s)}^{s} |H''(Y(r)) | \diff r \diff s \leq \frac{1}{2} L_{H} t_{m} \Delta t,
\end{equation*}
where we used that $H$ and $H',H''$ are uniformly bounded by $L_{H} +|\mu|$ and $L_{H}$, respectively, by Proposition~\ref{prop:H}. The third term on the right hand side of equation~\eqref{eq:Hito21} is a random variable with finite $p$th moments for every $p \geq 1$, see Appendix~\ref{sec:appA} for a proof. Inserting everything back into equation~\eqref{eq:LSerr1} and estimating $m \Delta t = t_{m} \leq T$ gives us
\begin{equation*}
\begin{split}
|Y^{LS}_{m}-Y(t_{m})| &\leq \left( L_{H} + |\mu| \right) L_{H} T \Delta t + \left( L_{H} + |\mu| \right) L_{H} T \Delta t + \frac{1}{2} L_{H} T \Delta t \\ &+ \left| \int_{0}^{T} \int_{\ell(s)}^{s} H'(Y(r)) \diff B(r) \diff s \right| + L_{H} \Delta t \sum_{k=0}^{m-1} |Y^{LS}_{k} - Y(t_{k})|,
\end{split}
\end{equation*}
where we have tried to group the terms to make it clear where the terms come from. An application of a discrete Grönwall's lemma now yields
\begin{equation*}
\sup_{m=0,\ldots,M}|Y^{LS}_{m}-Y(t_{m})| \leq \Delta t \eta_{\Delta t} e^{L_{H} T}
\end{equation*}
where $\eta_{\Delta t}$ is the random variable
\begin{equation}\label{eq:eta}
\eta_{\Delta t} = 2 \left( L_{H} + |\mu| \right) L_{H} T + \frac{1}{2} L_{H} T + \left| \frac{1}{\Delta t} \int_{0}^{T} \int_{\ell(s)}^{s} H'(Y(r)) \diff B(r) \diff s \right|.
\end{equation}
Observe that the random variable $\eta_{\Delta t}$ is not finite almost sure uniformly in $\Delta t$. However, $\Delta t \eta_{\Delta t}$ is finite almost surely uniformly in $\Delta t$. To finish the proof we need to make sure that $\eta_{\Delta t}$ has finite $p$th moments for every $p \geq 1$, since then
\begin{equation*}
\left( \E \left[ \sup_{m=0,\ldots,M}|Y^{LS}_{m}-Y(t_{m})|^{p} \right] \right)^{\frac{1}{p}} \leq \Delta t \left( \E \left[ \eta_{\Delta t}^{p} \right] \right)^{\frac{1}{p}} e^{L_{H} T} = C(p,H,T,\mu) \Delta t.
\end{equation*}
The only random term in the definition of $\eta_{\Delta t}$ in equation~\eqref{eq:eta} is the third. An application of stochastic Fubini's theorem and the Burkholder-Davis-Gundy (BDG) inequality gives us that the third term indeed has finite $p$th moments
\begin{equation*}
\E \left[ \left| \frac{1}{\Delta t} \int_{0}^{T} \int_{\ell(s)}^{s} H'(Y(r)) \diff B(r) \diff s \right|^{p} \right] \leq C(p) L_{H} T^{p/2},
\end{equation*}
where the constant $C(p)$ here is the constant in the BDG inequality. We provide the proof of the above $p$th moment estimates in Appendix~\ref{sec:appA}. This concludes the proof.
\end{proof}

By applying Lemma~$2.1$ in \cite{MR2320830}, we obtain, as a corollary to Theorem~\ref{theo:main_v2}, almost sure pathwise convergence with rate $1-\epsilon$ for every $\epsilon>0$. 
\begin{corollary}\label{cor:as_conv}
Under the same assumptions and notation as in Theorem~\ref{theo:main_v2}, there exists for every $\epsilon>0$ a random variable $\zeta_{\epsilon}$, with $\E \left[ |\zeta_{\epsilon}|^{p} \right] < \infty$ for every $p \geq 1$, such that
\begin{equation*}
\sup_{m=0, \ldots, M} |X^{LS}_{m}-X(t_{m})| \leq \zeta_{\epsilon} \Delta t^{1 - \epsilon}
\end{equation*}
almost surely.
\end{corollary}

\section{A boundary-preserving integrator}\label{sec:scheme_v2}
Here we extend the results in the previous section to the case where we do not assume that we can integrate the nonlinear ODE in~\eqref{eq:ODE-LampMain} exactly. More precisely, instead of assuming that there exists an explicit solution to the nonlinear ODE in~\eqref{eq:ODE-Lamp} we assume that we can integrate it using some deterministic integrator with local error of order $\Delta t^{2}$.
\begin{assumption}\label{ass:localInt}
For any $y_{0} \in \mathbb{R}$, we can find an approximation $\hat{y}$ to the exact solution $y$ of
\begin{equation}\label{eq:ODE-LampMain2}
\left\lbrace
\begin{aligned}
& \frac{\diff y(t)}{\diff t} = H(y(t)),\\ 
& y(0) = y_{0},
\end{aligned}
\right.
\end{equation}
with local error of order $\Delta t^{2}$ on the interval $[0,\Delta t]$; that is,
\begin{equation*}
\sup_{t \in [0,\Delta t]} |\hat{y}(t) - y(t)| \leq K \sqrt{1 + |y_{0}|^{2}} \Delta t^{2}
\end{equation*}
for some parameter $K>0$.
\end{assumption}
Observe that the classical Euler scheme and the exact solution $y$ both satisfies Assumption~\ref{ass:localInt}. The latter means that the setting in Section~\ref{sec:scheme} also fits into the setting of the current section. At this point, we also mention that we in this section for the convergence result in Theorem~\ref{theo:main_v3} will impose a time-step restriction; that is, we will assume that there exists $\gamma>0$ such that $K T \Delta t \leq \gamma$. Observe that this is in contrast to Section~\ref{sec:scheme} where no time-step restriction was imposed. The main reason for the time-step restriction in this section is to obtain estimates where the only dependence on $\Delta t$ is in the convergence rate. We only use the time-step restriction in the concluding lines of the proofs.

\subsection{Description of the integrator}\label{sec:schemeconstruction_v2}
Let $\hat{Y}^{LS}_{0} = \Phi(x_{0})$ and suppose that the approximation $\hat{Y}^{LS}_{m}$ at time point $t_{m} = m \Delta t$ is given. We let $\hat{y}_{m}(t)$, for $t \in [t_{m},t_{m+1}]$, be the approximate solution given in Assumption~\ref{ass:localInt} to equation~\eqref{eq:ODE-LampMain2} starting at $t_{m}$ with initial value $\hat{Y}^{LT}_{m}$ and we let the scheme at the next time point $t_{m+1} = (m+1) \Delta t$ be defined by
\begin{equation}\label{eq:LampSplitDisc2}
\hat{Y}^{LS}_{m+1} = \hat{Z}_{m}(t_{m+1}) \equiv \hat{y}_{m}(t_{m+1}) + \mu \Delta t + B(t_{m+1}) - B(t_{m})
\end{equation}
where $\hat{Z}_{m}(t)$, for $t \in [t_{m},t_{m+1}]$, solves
\begin{equation}\label{eq:SDE-main2}
\left\lbrace
\begin{aligned}
& \diff \hat{Z}_{m}(t) = \mu \diff t + \diff B(t),\\ 
& \hat{Z}_{m}(t_{m}) = \hat{y}_{m}(t_{m+1}).
\end{aligned}
\right.
\end{equation}
We define $\hat{X}^{LS}_{m+1} = \Phi^{-1}\left(\hat{Y}^{LS}_{m+1} \right)$ as the approximation of the solution of equation~\eqref{eq:SDEmain} at the time point $t_{m+1}$. We remark that the definition~\eqref{eq:LampSplitDisc2} implies the following representation formula for $\hat{Y}^{LS}$
\begin{equation}\label{eq:expresshatYLS}
\begin{split}
\hat{Y}^{LS}_{m} &= Y(0) + \sum_{k=0}^{m-1} \left( \hat{y}_{k}(t_{k+1}) - \tilde{y}_{k}(t_{k+1}) \right) \\ &+ \sum_{k=0}^{m-1} \int^{t_{k+1}}_{t_{k}} \left( H(\tilde{y}_{k}(s)) + \mu \right) \diff s + B(t_{m})
\end{split}
\end{equation}
where $\tilde{y}_{k}(t)$, for $t \in [t_{k},t_{k+1}]$, is the exact solution to the nonlinear ODE in equation~\eqref{eq:ODE-LampMain2} starting at $t_{k}$ with initial value $\hat{Y}^{LS}_{k}$. Compare the formula~\eqref{eq:expresshatYLS} to the formula in equation~\eqref{eq:expressYLS}, the additional term in formula~\eqref{eq:expresshatYLS} comes from the deterministic integrator in Assumption~\ref{ass:localInt}. The representation in equation~\eqref{eq:expresshatYLS} follows by repeatedly using equation~\eqref{eq:LampSplitDisc2} as follows: First add and subtract $\tilde{y}_{m}(t_{m+1})$ to equation~\eqref{eq:LampSplitDisc2}
\begin{equation*}
\begin{split}
\hat{Y}^{LS}_{m+1} &= \hat{y}_{m}(t_{m+1}) + \mu \Delta t + B(t_{m+1}) - B(t_{m}) \\ &= \hat{y}_{m}(t_{m+1}) - \tilde{y}_{m}(t_{m+1}) + \mu \Delta t + B(t_{m+1}) - B(t_{m}) + \tilde{y}_{m}(t_{m+1})
\end{split}
\end{equation*}
and then insert that
\begin{equation*}
\tilde{y}_{m}(t_{m+1}) = \hat{Y}^{LS}_{m} + \int_{t_{m}}^{t_{m+1}} H(\tilde{y}_{m}(s)) \diff s
\end{equation*}
to obtain
\begin{equation*}
\begin{split}
\hat{Y}^{LS}_{m+1} &= \hat{Y}^{LS}_{m} + \hat{y}_{m}(t_{m+1}) - \tilde{y}_{m}(t_{m+1}) + \int_{t_{m}}^{t_{m+1}} \left( H(\tilde{y}_{m}(s)) + \mu \right) \diff s \\ &+ B(t_{m+1}) - B(t_{m}).
\end{split}
\end{equation*}
Repeating the above one more time gives
\begin{equation*}
\begin{split}
\hat{Y}^{LS}_{m+1} &= \hat{Y}^{LS}_{m-1} + \hat{y}_{m}(t_{m+1}) - \tilde{y}_{m}(t_{m+1}) + \hat{y}_{m-1}(t_{m}) - \tilde{y}_{m-1}(t_{m}) \\ &+ \int_{t_{m}}^{t_{m+1}} \left( H(\tilde{y}_{m}(s)) + \mu \right) \diff s + \int_{t_{m-1}}^{t_{m}} \left( H(\tilde{y}_{m-1}(s)) + \mu \right) \diff s \\ &+ B(t_{m+1}) - B(t_{m-1})
\end{split}
\end{equation*}
and repeating $m$ times gives the desired representation formula. 

As in Section~\ref{sec:scheme}, the above constructed scheme is boundary-preserving which we state below in Proposition~\ref{propo:BP-scheme_v2}. The proof is the same as the proof of Proposition~\ref{propo:BP-scheme}.
\begin{proposition}\label{propo:BP-scheme_v2}
Let $M\in\N$, $T > 0$, $\Delta t = T/M$ and let $x_{0} \in \mathring{D}$.
Suppose that Assumptions~\ref{ass:f},~\ref{ass:g},~\ref{ass:fg} and~\ref{ass:localInt} are satisfied.
Let $\hat{Y}^{LS}$ be given by the splitting scheme in equation~\eqref{eq:LampSplitDisc} and let $\hat{X}^{LS} = \Phi^{-1}(\hat{Y}^{LS})$ be the numerical approximation of the original SDE in~\eqref{eq:SDEmain}. Then
\begin{equation*}
\hat{X}^{LS}_{m} \in \mathring{D},
\end{equation*}
almost surely, for every $m \in 0,\ldots, M$.
\end{proposition}

\subsection{Convergence result}\label{sec:results_v2}
We start with proving that $\sup_{m=0,\ldots,M} |\hat{Y}^{LS}|$ is almost surely bounded. This is the content of the following proposition and is proved by a standard Grönwall argument. Observe that moment bounds of $\sup_{m=0,\ldots,M} |\hat{Y}^{LS}_{m}|$ would have been enough for the proof of the $L^{p}(\Omega)$-convergence in Theorem~\ref{theo:main_v3}. Notice that the moment bounds where not needed in Section~\ref{sec:schemeconstruction} because of the assumption of an exact solution to the ODEs in equation~\eqref{eq:ODE-LampMain}. 
\begin{proposition}\label{prop:LSbound}
Let $M\in\N$, $T > 0$, $\Delta t = T/M$ and let $x_{0} \in \mathring{D}$. Suppose Assumptions~\ref{ass:f},~\ref{ass:g},~\ref{ass:fg} and~\ref{ass:localInt} are satisfied. Let $\hat{Y}^{LS}$ be given by the splitting scheme in equation~\eqref{eq:LampSplitDisc2} with initial value $\hat{Y}^{LS}_{0} = Y(0) = \Phi(x_{0})$. Then it holds
\begin{equation*}
\sup_{m=0,\ldots,M} |\hat{Y}^{LS}_{m}| \leq \left( K T \Delta t + L_{H} T + |Y(0)| + \sup_{t \in [0,T]} |B(t)| \right) e^{K T \Delta t}
\end{equation*}
almost surely.
\end{proposition}
We remark the time-step restriction $K T \Delta t \leq \gamma$, for some constant $\gamma>0$, in the convergence result in Theorem~\ref{theo:main_v3} below comes from Propostion~\ref{prop:LSbound}. Observe that Proposition~\ref{prop:LSbound} provides a bound on $\sup_{m=0,\ldots,M} |\hat{Y}^{LS}_{m}|$ that is independent of $\mu$.
\begin{proof}[Proof of Proposition~\ref{prop:LSbound}]
We first apply the triangle inequality to the representation formula for $\hat{Y}^{LS}$ in equation~\eqref{eq:expresshatYLS} to estimate
\begin{equation*}
\begin{split}
| \hat{Y}^{LS}_{m} | &\leq |Y(0)| + \sum_{k=0}^{m-1} |\hat{y}_{k}(t_{k+1}) - \tilde{y}_{k}(t_{k+1})| \\ &+ \sum_{k=0}^{m-1} \int_{t_{k}}^{t_{k+1}} | H(\tilde{y}_{k}(s))+\mu| \diff s + |B(t_{m})|.
\end{split}
\end{equation*}
Next, since we have the estimate
\begin{equation*}
|\hat{y}_{k}(t_{k+1}) - \tilde{y}_{k}(t_{k+1})| \leq K \sqrt{1 + |\hat{Y}^{LS}_{k}|^{2}} \Delta t^{2} \leq K \left( 1 + |\hat{Y}^{LS}_{k}| \right) \Delta t^{2}
\end{equation*}
by Assumption~\ref{ass:localInt} and since $|H(x)+\mu|$ is uniformly bounded by $L_{H}$ by Proposition~\ref{prop:H}, we can further estimate
\begin{equation*}
\begin{split}
| \hat{Y}^{LS}_{m} | &\leq |Y(0)| + K \Delta t^{2} \sum_{k=0}^{m-1} \left( 1 + |\hat{Y}^{LS}_{k}| \right) + L_{H} t_{m} + |B(t_{m})| \\ &\leq |Y(0)| + \left(K \Delta t + L_{H} \right) T + K \Delta t^{2} \sum_{k=0}^{m-1} |\hat{Y}^{LS}_{k}| + \sup_{t \in [0,T]} |B(t)|
\end{split}
\end{equation*}
where we used that $m \Delta t = t_{m} \leq T$. We then apply Grönwall's lemma to conclude that
\begin{equation}\label{eq:Ymest}
|\hat{Y}^{LS}_{m}| \leq \left( K T \Delta t + L_{H} T + |Y(0)| + \sup_{t \in [0,T]} |B(t)| \right) e^{K T \Delta t},
\end{equation}
almost surely. Applying $\sup_{m=0,\ldots,M}$ to equation~\eqref{eq:Ymest} gives the desired result.
\end{proof}
We now prove the main convergence theorem in this section. The proof is similar to the proof of Theorem~\ref{theo:main_v2}, but with the different representation formula in equation~\eqref{eq:expresshatYLS}. We provide both proofs for completeness.
\begin{theorem}\label{theo:main_v3}
Let $M\in\N$, $T > 0$, $\Delta t = T/M$ and let $x_{0} \in \mathring{D}$.
Suppose Assumptions~\ref{ass:f},~\ref{ass:g},~\ref{ass:fg} and ~\ref{ass:localInt} are satisfied and that there exists some constant $\gamma > 0$ such that $K T \Delta t \leq \gamma$. Let $\hat{X}^{LS} = \Phi^{-1}(Y^{LS})$, where $\hat{Y}^{LS}$ is defined by the splitting scheme in equation~\eqref{eq:LampSplitDisc2}, and let $X$ be the exact solution of the considered SDE in equation~\eqref{eq:SDEmain}. Then, for every $p \geq 1$, it holds
\begin{equation*}
\left( \E \left[ \sup_{m=0, \ldots, M} | \hat{X}^{LS}_{m} - X(t_{m}) |^{p} \right] \right)^{\frac{1}{p}} \leq C(p,\Phi^{-1},H,K,T,\mu,\gamma) \Delta t.
\end{equation*}
\end{theorem}
The dependence of the constant $C(p,\Phi^{-1},H,K,T,\mu,\gamma)$ on the parameters can be found in the proof of Theorem~\ref{theo:main_v3}.
\begin{proof}[Proof of Theorem~\ref{theo:main_v3}]
Since $\Phi^{-1}$ is globally Lipschitz continuous by Proposition~\ref{prop:phiinv}, the proof reduces to proving the corresponding $L^{p}(\Omega)$-estimate for $|\hat{Y}^{LS}_{m}-Y(t_{m})|$:
\begin{equation*}
| \hat{X}^{LS}_{m} - X(t_{m}) | = | \Phi^{-1}(\hat{Y}^{LS}_{m}) - \Phi^{-1}(Y(t_{m})) | \leq L_{\Phi^{-1}} | \hat{Y}^{LS}_{m} - Y(t_{m}) |.
\end{equation*}
The rest of the proof is based on comparing the following representation formulas for $\hat{Y}^{LS}_{m}$ and $Y(t_{m})$:
\begin{equation*}
\begin{split}
\hat{Y}^{LS}_{m} &= Y(0) + \sum_{k=0}^{m-1} \hat{y}_{k}(t_{k+1}) - \tilde{y}_{k}(t_{k+1}) \\ &+ \sum_{k=0}^{m-1} \int^{t_{k+1}}_{t_{k}} \left( H(\tilde{y}_{k}(s)) + \mu \right) \diff s + B(t_{m}),
\end{split}
\end{equation*}
introduced in equation~\eqref{eq:expresshatYLS}, and the integral representation of $Y$ that we re-write on similar form
\begin{equation*}
\begin{split}
Y(t_{m}) &= Y(0) + \int^{t_{m}}_{0} \left( H(Y(s)) + \mu) \right) \diff s + B(t_{m}) \\ &= Y(0) + \sum_{k=0}^{m-1} \int_{t_{k}}^{t_{k+1}} \left( H(Y(s)) + \mu \right) \diff s + B(t_{m})
\end{split}
\end{equation*}
at $t=t_{m}$. Therefore
\begin{equation}\label{eq:LSerr}
\begin{split}
|\hat{Y}^{LS}_{m}-Y(t_{m})| &\leq \sum_{k=0}^{m-1} |\hat{y}_{k}(t_{k+1})-\tilde{y}_{k}(t_{k+1})| \\ &+ \left|\sum_{k=0}^{m-1} \int_{t_{k}}^{t_{k+1}} H\left(\tilde{y}_{k}(s)\right) - H\left(Y(s)\right) \diff s \right|,
\end{split}
\end{equation}
We first estimate the first term in equation~\eqref{eq:LSerr}. By Assumption~\ref{ass:localInt}, we can bound
\begin{equation*}
\begin{split}
|\hat{y}_{k}(t_{k+1})-\tilde{y}_{k}(t_{k+1})| &\leq K \sqrt{1 + |\hat{Y}^{LS}_{k}|^{2}} \Delta t^{2} \\ &\leq K \left(1 + |\hat{Y}^{LS}_{k}| \right) \Delta t^{2} \\ &\leq K \left(1 + \sup_{k=0,\ldots,M} |\hat{Y}^{LS}_{k}| \right) \Delta t^{2},
\end{split}
\end{equation*}
almost surely. Recall that $\sup_{k=0,\ldots,M} |\hat{Y}^{LS}_{k}|$ is finite, almost surely, by Proposition~\ref{prop:LSbound} and that the bound can be made independent of $\Delta t$ because of the time-step restriction $K T \Delta t \leq \gamma$. Thus
\begin{equation*}
\sum_{k=0}^{m-1} |\hat{y}_{k}(t_{k+1})-\tilde{y}_{k}(t_{k+1})| \leq K t_{m} \left(1 + \sup_{k=0,\ldots,M} |\hat{Y}^{LS}_{k}| \right) \Delta t.
\end{equation*}
Let us now estimate the second term on the right hand side of equation~\eqref{eq:LSerr}. We first split the integral as follows
\begin{equation}\label{eq:err}
\begin{split}
\int_{t_{k}}^{t_{k+1}} H \left( \tilde{y}_{k}(s) \right) - H \left( Y(s) \right) \diff s &= \int_{t_{k}}^{t_{k+1}} H \left( \tilde{y}_{k}(s) \right) - H \left( \hat{Y}^{LS}_{k} \right) \diff s \\ &+ \int_{t_{k}}^{t_{k+1}} H \left( \hat{Y}^{LS}_{k} \right) - H \left( Y(t_{k}) \right) \diff s \\ &+ \int_{t_{k}}^{t_{k+1}} H \left( Y(t_{k}) \right) - H \left( Y(s) \right) \diff s
\end{split}
\end{equation}
and consider each term separately. Observe that the second term on the right hand side of equation~\eqref{eq:err} is the term that we need for Grönwall's lemma after the simplification
\begin{equation*}
\int_{t_{k}}^{t_{k+1}} H \left( \hat{Y}^{LS}_{k} \right) - H \left( Y(t_{k}) \right) \diff s = \Delta t \left( H \left( \hat{Y}^{LS}_{k} \right) - H \left( Y(t_{k}) \right) \right) 
\end{equation*}
and after estimating it
\begin{equation*}
\left| \int_{t_{k}}^{t_{k+1}} H \left( \hat{Y}^{LS}_{k} \right) - H \left( Y(t_{k}) \right) \diff s \right| \leq L_{H} \Delta t \left| \hat{Y}^{LS}_{k} - Y(t_{k}) \right|
\end{equation*}
using that $H$ is globally Lipschitz continuous by Proposition~\ref{prop:H}. Recall now that $\tilde{y}_{k}(s) = \hat{Y}^{LS}_{k} + \int_{t_{k}}^{s} H(\tilde{y}_{k}(r)) \diff r$, for $s \in [t_{k},t_{k+1}]$, is the exact solution to the ODE in equation~\eqref{eq:ODE-LampMain2} starting at $t_{k}$ with initial value $\hat{Y}^{LS}_{k}$. Therefore, the first term on the right hand side in equation~\eqref{eq:err} can be estimated by
\begin{equation*}
\begin{split}
\left| \int_{t_{k}}^{t_{k+1}} H \left(\tilde{y}_{k}(s) \right) - H \left( \hat{Y}^{LS}_{k} \right) \diff s \right| &\leq L_{H} \int_{t_{k}}^{t_{k+1}} |\tilde{y}_{k}(s) - \hat{Y}^{LS}_{k} | \diff s \\ &\leq L_{H} \int_{t_{k}}^{t_{k+1}} \int_{t_{k}}^{s} |H(\tilde{y}_{k}(r))| \diff r \diff s
\end{split}
\end{equation*}
since $H$ is Lipschitz continuous and by using that $H$ is uniformly bounded by $L_{H} + |\mu|$ by Proposition~\ref{prop:H} we can further estimate
\begin{equation*}
L_{H} \int_{t_{k}}^{t_{k+1}} \int_{t_{k}}^{s} |H(\tilde{y}_{k}(r))| \diff r \diff s \leq L_{H} \left( L_{H} + |\mu| \right) \Delta t^{2}. 
\end{equation*}
It remains to consider the third term in equation~\eqref{eq:err}
\begin{equation*}
\int_{t_{k}}^{t_{k+1}} H \left( Y(t_{k}) \right) - H \left( Y(s) \right) \diff s.
\end{equation*}
Recall now that $H \in C^{2} \left( \mathbb{R} \right)$ by Proposition~\ref{prop:H}. Thus we can use Itô's formula as such:
\begin{equation*}
\begin{split}
H(Y(s)) &= H(Y(t_{k})) + \int_{t_{k}}^{s} H(Y(r)) H'(Y(r)) \diff r \\ &+ \frac{1}{2} \int_{t_{k}}^{s} H''(Y(r))  \diff r + \int_{t_{k}}^{s} H'(Y(r)) \diff B(r)
\end{split}
\end{equation*}
to see that
\begin{equation}\label{eq:Hito}
\begin{split}
\int_{t_{k}}^{t_{k+1}} H \left( Y(t_{k}) \right) - H \left( Y(s) \right) \diff s  &= \int_{t_{k}}^{t_{k+1}} \int_{t_{k}}^{s} H(Y(r)) H'(Y(r)) \diff r \diff s \\ &+ \frac{1}{2} \int_{t_{k+1}}^{t_{k+1}} \int_{t_{k}}^{s} H''(Y(r)) \diff r \diff s  \\ &+ \int_{t_{k}}^{t_{k+1}} \int_{t_{k}}^{s} H'(Y(r)) \diff B(r) \diff s.
\end{split}
\end{equation}
To obtain convergence order $1$, we have to sum over $k=0,\ldots,m-1$ before estimating the separate terms in equation~\eqref{eq:Hito}
\begin{equation}\label{eq:Hito2}
\begin{split}
\sum_{k=0}^{m-1} \int_{t_{k}}^{t_{k+1}} H \left( Y(t_{k}) \right) - H \left( Y(s) \right) \diff s  &= \int_{0}^{t_{m}} \int_{\ell(s)}^{s} H(Y(r)) H'(Y(r)) \diff r \diff s \\ &+ \frac{1}{2} \int_{0}^{t_{m}} \int_{\ell(s)}^{s} H''(Y(r)) \diff r \diff s \\ &+ \int_{0}^{t_{m}} \int_{\ell(s)}^{s} H'(Y(r)) \diff B(r) \diff s,
\end{split}
\end{equation}
where we recall that $\ell(s) = t_{k}$ whenever $s \in [t_{k},t_{k+1})$. The first two terms on the right hand side of equation~\eqref{eq:Hito2} can estimated by
\begin{equation*}
\begin{split}
\left| \int_{0}^{t_{m}} \int_{\ell(s)}^{s} H(Y(r)) H'(Y(r)) \diff r \diff s \right| &\leq \int_{0}^{t_{m}} \int_{\ell(s)}^{s} | H(Y(r)) H'(Y(r))| \diff r \diff s \\ &\leq L_{H} \left( L_{H} + |\mu| \right) t_{m} \Delta t
\end{split}
\end{equation*}
and
\begin{equation*}
\begin{split}
\left| \frac{1}{2} \int_{0}^{t_{m}} \int_{\ell(s)}^{s} H''(Y(r)) \diff r \diff s \right| &\leq \frac{1}{2} \int_{0}^{t_{m}} \int_{\ell(s)}^{s} |H''(Y(r)) | \diff r \diff s \\ &\leq \frac{1}{2} L_{H} t_{m} \Delta t,
\end{split}
\end{equation*}
where we used that $H$ and $H',H''$ are uniformly bounded by $L_{H} + |\mu|$ and $L_{H}$, respectively, by Proposition~\ref{prop:H}. The third term on the right hand side of equation~\eqref{eq:Hito2} is a random variable with finite $p$th moments for every $p \geq 1$, see Appendix~\ref{sec:appA} for a proof. Inserting everything back into equation~\eqref{eq:LSerr} and estimating $m \Delta t = t_{m} \leq T$ gives us
\begin{equation*}
\begin{split}
|\hat{Y}^{LS}_{m}-Y(t_{m})| &\leq K T \left(1 + \sup_{k=0,\ldots,M} |\hat{Y}^{LS}_{k}| \right) \Delta t + 2 L_{H} \left( L_{H} + |\mu| \right) T \Delta t \\ &+ \frac{1}{2} L_{H} T \Delta t + \left| \int_{0}^{T} \int_{\ell(s)}^{s} H'(Y(r)) \diff B(r) \diff s \right| \\ &+ L_{H} \Delta t \sum_{k=0}^{m-1} |\hat{Y}^{LS}_{k} - Y(t_{k})|,
\end{split}
\end{equation*}
where we have tried to group the terms to make it clear where the terms come from. An application of a discrete Grönwall's lemma followed by applying $\sup_{m=0,\ldots,M}$ now yields
\begin{equation}\label{eq:asconv}
\sup_{m=0,\ldots,M}|\hat{Y}^{LS}_{m}-Y(t_{m})| \leq \Delta t \eta_{\Delta t} e^{L_{H} T}
\end{equation}
where $\eta_{\Delta t}$ is the random variable
\begin{equation*}
\begin{split}
\eta_{\Delta t} &= K T \left(1 + \sup_{k=0,\ldots,M} |\hat{Y}^{LS}_{k}| \right) + 2 L_{H} \left( L_{H} + |\mu| \right) T + \frac{1}{2} L_{H} T \\ &+ \left| \frac{1}{\Delta t} \int_{0}^{T} \int_{\ell(s)}^{s} H'(Y(r)) \diff B(r) \diff s \right|.
\end{split}
\end{equation*}
Note that the random variable $\eta_{\Delta t}$ is not finite almost sure uniformly in $\Delta t$, but $\Delta t \eta_{\Delta t}$ is finite almost surely uniformly in $\Delta t$. To finish the proof we need to make sure that $\eta_{\Delta t}$ has finite $p$th moments for every $p \geq 1$, since then
\begin{equation*}
\begin{split}
\left( \E \left[ \sup_{m=0,\ldots,M}|\hat{Y}^{LS}_{m}-Y(t_{m})|^{p} \right] \right)^{\frac{1}{p}} &\leq \Delta t \left( \E \left[ \eta_{\Delta t}^{p} \right] \right)^{\frac{1}{p}} e^{L_{H} T} \\ &= C(p,H,K,T,\mu,\gamma) \Delta t.
\end{split}
\end{equation*}
By Proposition~\ref{prop:LSbound}, the first three terms constituting $\eta_{\Delta t}$ are almost surely bounded and thus in particular have finite $p$th moments for every $p \geq 1$. An application of the stochastic Fubini's theorem followed by the Burkholder-Davis-Gundy (BDG) inequality gives us that the remaining term has finite $p$th moments
\begin{equation*}
\E \left[ \left| \frac{1}{\Delta t} \int_{0}^{T} \int_{\ell(s)}^{s} H'(Y(r)) \diff B(r) \diff s \right|^{p} \right] \leq C(p) L_{H} T^{p/2},
\end{equation*}
where the constant $C(p)$ is the constant in the BDG inequality. We provide the proof of the above $p$th moment estimate in Appendix~\ref{sec:appA}. This concludes the proof.
\end{proof}

In the same manner as almost sure pathwise convergence of $X^{LS}$ towards $X$ given in Corollary~\ref{cor:as_conv}, we obtain, as a corollary to Theorem~\ref{theo:main_v3}, almost sure pathwise convergence of $\hat{X}^{LS}$ towards $X$ with rate $1-\epsilon$ for every $\epsilon>0$.
\begin{corollary}\label{cor:as_conv2}
Under the same assumptions and notation as in Theorem~\ref{theo:main_v3}, there exists for every $\epsilon>0$ a random variable $\zeta_{\epsilon}$, with $\E \left[ |\zeta_{\epsilon}|^{p} \right] < \infty$ for every $p \geq 1$, such that
\begin{equation*}
\sup_{m=0, \ldots, M} |\hat{X}^{LS}_{m}-X(t_{m})| \leq \zeta_{\epsilon} \Delta t^{1 - \epsilon}
\end{equation*}
almost surely.
\end{corollary}

\section{Numerical experiments}\label{sec:num}
In this section we provide numerical experiments to support and verify the theoretical results in Section~\ref{sec:scheme} and Section~\ref{sec:scheme_v2}. We introduce a noise scale parameter $\lambda > 0$ in the considered SDE in equation~\eqref{eq:SDEmain} to avoid the need to run the numerical experiments for large time horizon $T>0$. Hence, we consider the SDE
\begin{equation}\label{eq:SDEnum}
\left\lbrace
\begin{aligned}
& \diff X(t) = f(X(t)) \diff t + \lambda g(X(t)) \diff B(t),\ t \in (0,T],\\ 
& X(0) = x_{0} \in \mathring{D}.
\end{aligned}
\right.
\end{equation}
For each numerical experiment, we use $T=0.4$ or $T=1$ and we use the initial value $x_{0}=0.9$ or we let $x_{0}$ be uniformly distributed on $\mathring{D}$.

We provide numerical results for three choices of the drift and diffusion coefficients $f$ and $g$:
\begin{itemize}
\item the Susceptible–Infected–Susceptible (SIS) SDE with $f(x) = x-x^2$ and $g(x) = x-x^2$,
\item the Nagumo SDE with $f(x) = -x (1-x) (1-x)$ and $g(x) = - x + x^2$,
\item an Allen--Cahn type SDE with $f(x) = x - x^3$ and $g(x) = 1 - x^2$.
\end{itemize}
The first two examples above satisfy the assumptions of Section~\ref{sec:scheme_v2} and the Allen-Cahn typ SDE satisfies the assumptions in Section~\ref{sec:scheme}, see Appendix~\ref{sec:app} for details. The SIS SDE \cite{MR4220738, MR3015083, GraySIS, MR4274899, Yang2021FirstOS, MR4268206} is a model for the spread of epidemics and is also used in gene frequency modelling (for example Wright--Fisher diffusion). We refer the interested reader to, for example, \cite{MR3015083, GraySIS} for detailed descriptions of such models. The Nagumo SDE \cite{MR3308418, MCKEAN1970209} and the Allen--Cahn type SDE \cite{ALLEN19791085, MR3986273, Funaki1995TheSL, MR3308418} are motivated by a finite difference space discretisation of the corresponding stochastic partial differential equations (SPDEs). The stochastic Nagumo equation is a stochastic model for the voltage in the axon of a neuron. The stochastic Allen--Cahn equation is a stochastic model for the time evolution of the interface between two phases of a material. We refer the interested reader to \cite{MR3308418} for details on these SPDEs. 

Recall that $M \in \mathbb{N}$ and $\Delta t = T/M$ denote the number of subintervals used to partition $[0,T]$ and the time-step, respectively, of the numerical schemes. We denote by $\Delta B_{m} = B(t_{m+1}) - B(t_{m})$ the increment of the Brownian motion over the interval $[t_{m},t_{m+1}] = [ m \Delta t, (m+1) \Delta t]$. We compare boundary-preservation of the proposed Lamperti-splitting schemes, denoted LS  and $X^{LS}_{m}$ below, as defined by $\Phi^{-1}(Y^{LS}_{m+1})$ and equation~\eqref{eq:LampSplitDisc} or by $\Phi^{-1}(\hat{Y}^{LS}_{m+1})$ and equation~\eqref{eq:LampSplitDisc2} to boundary-preservation of the following integrators for SDEs:
\begin{itemize}
\item the Euler--Maruyama scheme (denoted EM below), see for instance \cite{MR1214374}
$$
X^{\rm EM}_{m+1}=X^{\rm EM}_m+ f(X^{\rm EM}_{m})\Delta t +\lambda g(X^{\rm EM}_m) \Delta B_{m},
$$
\item the semi-implicit Euler--Maruyama scheme (denoted SEM below), see for instance \cite{MR1214374}
$$
X^{\rm SEM}_{m+1}=X^{\rm SEM}_m+ f(X^{\rm SEM}_{m+1})\Delta t + \lambda g(X^{\rm SEM}_m) \Delta B_{m},
$$
\item the tamed Euler scheme (denoted TE below), see for instance \cite{MR2985171, MR3543890}
$$
X^{\rm TE}_{m+1}=X^{\rm TE}_m+ f_{M}(X^{\rm TE}_{m})\Delta t + \lambda g_{M}(X^{\rm TE}_m) \Delta B_{m},
$$
where
$$
f_{M}(x) = \frac{f(x)}{1 + M^{-1/2} |f(x)| + M^{-1/2} |\lambda g(x)|^{2}}
$$
$$
g_{M}(x) = \frac{g(x)}{1 + M^{-1/2} |f(x)| + M^{-1/2} |\lambda g(x)|^{2}}.
$$
\end{itemize}

We consider values of $\lambda>0$ which illustrate that EM, SEM, and TE are not boundary-preserving. Note that $g$ has to be a map $\mathbb{R} \to \mathbb{R}$ (see Assumption~\ref{ass:g}) for the EM, SEM and TE schemes to be well-defined, but this is not a problem as $g$ is a polynomial defined on $\mathbb{R}$ in the examples we provide.

For the above three SDEs, we provide numerical experiments illustrating the boundary-preservation as well as the $L^{2}(\Omega)$-convergence of order $1$ of the LS scheme as derived in Sections~\ref{sec:scheme} and~\ref{sec:scheme_v2}. We present boundary-preservation in tables displaying the proportion out of $100$ simulated sample paths that contain only values in the domain $\mathring{D}$ and we present, in loglog plots, the $L^{2}(\Omega)$-errors
\begin{equation*}
\left( \E \left[ \sup_{m=1,\ldots,M} |X^{LS}_{m} - X^{ref}_{m}|^{2} \right] \right)^{1/2}
\end{equation*}
over the time grid points $\{ t_{m}:\ m=1,\ldots, M \}$. Recall that we let $X^{LS}$ denote either the semi-analytic LS scheme in Section~\ref{sec:scheme} or the LS scheme in Section~\ref{sec:scheme_v2}. The reference solution $X^{ref}$ is computed using the LS scheme with time-step $\Delta t^{ref} = 10^{-7}$. We have also computed the $L^{2}(\Omega)$-errors for the LS scheme with the reference solution computed using the Lamperti EM scheme (see, e.g., \cite{MR4274899, MR4268206}) and the Lamperti SEM scheme (see, e.g., \cite{MR3006996, MR3248050}), respectively, and obtained similar results. For approximation of the expectations for the $L^{2}(\Omega)$-convergence, we use $300$ simulated samples. We have checked that $300$ simulated samples is sufficient for the Monte Carlo error to be negligible.

For ease of presentation, lengthy and complicated formulas are collected in Appendix~\ref{sec:app}.

\begin{example}[SIS SDE]
Consider the SIS epidemic model given by
\begin{equation*}
\diff X(t) = X(t) \left(1- X(t) \right) \diff t + \lambda X(t)(1-X(t)) \diff B(t)
\end{equation*}
with initial value $X(0)=x_{0} \in (0,1)$;
that is, $f(x) = x(1-x)$ and $g(x) = \lambda x(1-x)$ in the considered SDE in equation~\eqref{eq:SDEmain} are both quadratic. In this example we let $H(x) = \lambda^2 \Phi^{-1}(x)$ and $\mu = 1 - \lambda^2/2$, see Section~\ref{sec:app-SIS} for more details about the explicit formulas used for the implementation of the LS scheme for the SIS SDE. The exact solution $X$ takes values in $(0,1)$, since the inverse Lamperti transform
\begin{equation*}
\Phi^{-1}(x) = \frac{w_{0} e^{x}}{w_{0} e^{x} + (1-w_{0})}
\end{equation*}
takes values in $(0,1)$, for any $w_{0} \in (0,1)$. See Section~\ref{sec:BB} and Section~\ref{sec:app-SIS} for more details. We first provide, in Figure~\ref{num:SIS_path_comp} below, a plot showing sample paths where the comparison schemes EM, SEM and TE all leave the domain $(0,1)$.

\begin{figure}[htp]
\begin{center}
  \includegraphics[width=0.8\textwidth]{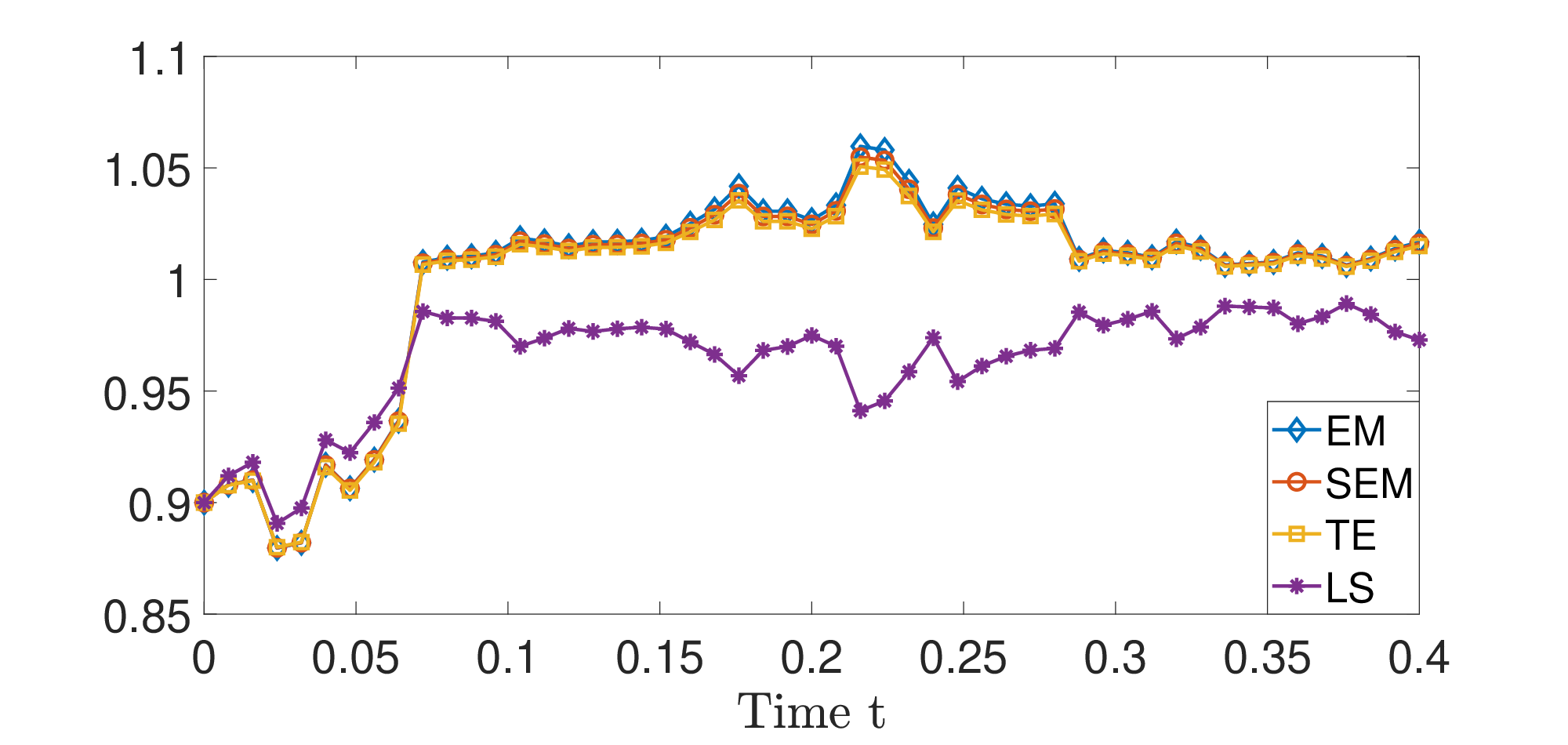}
  \caption{Path comparison of the EM, SEM, TE and LS schemes applied to the SIS SDE with parameters $\lambda = 4$, $x_{0} = 0.9$, $T = 0.4$ and $M=50$.}\label{num:SIS_path_comp}
  \end{center}
\end{figure}

In Table~\ref{tb:SIS}, we observe that the LS scheme preserves the domain $(0,1)$ of the SIS SDE while the integrators EM, SEM, and TE do not. As expected, the number of samples that preserve the domain $(0,1)$ for EM, SEM, and TE, respectively, decreases as $\lambda>0$ increases. In Table~\ref{tb:SIS}, we used $\Delta t=10^{-3}$, $T=1$, $N=100$ number of samples and $x_{0}$ uniformly distributed on $(0,1)$ for each sample.

\begin{table}[!htbp]
\begin{center}
\begin{tabular}{||c c c c c||} 
 \hline
 $\lambda$ & LS & EM & SEM & TE\\ [0.5ex] 
 \hline\hline
 $6$ & $100/100$ & $100/100$ & $100/100$ & $100/100$ \\ 
 \hline
 $7$ & $100/100$ & $94/100$ & $89/100$ & $92/100$ \\ 
 \hline
 $8$ & $100/100$ & $71/100$ & $63/100$ & $70/100$ \\ [1ex]
 \hline
\end{tabular}
\caption{Proportion of samples containing only values in $(0,1)$ out of $100$ simulated sample paths for the Lamperti-splitting scheme (LS), the Euler--Maruyama scheme (EM), the semi-implicit Euler--Maruyama scheme (SEM), and the tamed Euler scheme (TE) for the SIS SDE for different choices of $\lambda>0$. The parameters used are: $T=1$, $\Delta t = 10^{-3}$ and with $x_{0}$ uniformly distributed on $(0,1)$ for each sample. \label{tb:SIS}}
\end{center}
\end{table}
In Figure~\ref{num:SIS_conv} we present the $L^{2}(\Omega)$-errors for the same values of $\lambda$ as used in Table~\ref{tb:SIS}. The $L^{2}(\Omega)$-error rates in Figure~\ref{num:SIS_conv} agree with the rates obtained in Theorem~\ref{theo:main_v3}. We use $T=1$, $N=300$ number of samples to approximate the expected value and $x_{0}$ is uniformly distributed on $(0,1)$ for each sample in Figure~\ref{num:SIS_conv}.

\begin{figure}[htp]
\begin{center}
  \includegraphics[width=0.8\textwidth]{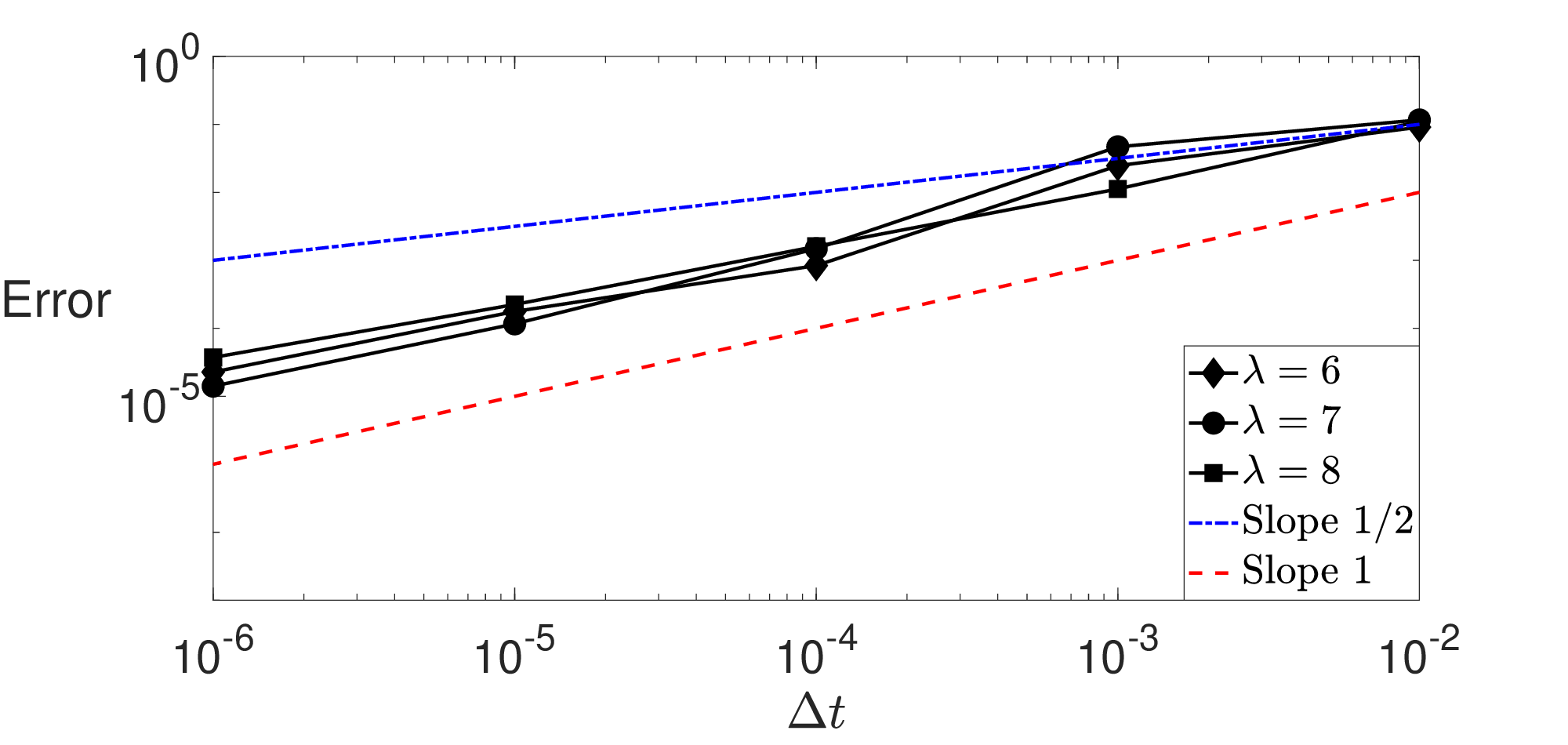}
  \caption{$L^{2}(\Omega)$-errors on the interval $[0,1]$ of the Lamperti-splitting scheme (LS) for the SIS SDE for different choices of $\lambda>0$ and reference lines with slopes $1/2$ and $1$. Averaged over $300$ samples.}\label{num:SIS_conv}
  \end{center}
\end{figure}

\end{example}

\begin{example}[Nagumo SDE]
Consider the Nagumo SDE given by
\begin{equation*}
\diff X(t) = - X(t)(1-X(t))(1-X(t)) \diff t - \lambda X(t) (1-X(t)) \diff B(t)
\end{equation*}
with initial value $X(0) = x_{0} \in (0,1)$; that is, $f(x) = - x (1-x)(1-x)$ is cubic and $g(x) = - \lambda x(1-x)$ is quadratic in the considered SDE in equation~\eqref{eq:SDEmain}. Here we let $H(x) = -(1 + \lambda^2) \Phi^{-1}(x)$ and $\mu = (1 + \frac{\lambda^2}{2})$, see Section~\ref{sec:app-NAG} for more details about explicit formulas used in the implementation of the LS scheme for the Nagumo SDE. As is derived in Section~\ref{sec:app-NAG}, the inverse Lamperti transform is given by
\begin{equation*}
\Phi^{-1}(x) = \frac{w_{0}}{(1-w_{0})e^{ x} + w_{0}}
\end{equation*}
and takes values in $(0,1)$, for any $w_{0} \in (0,1)$. Hence, by Section~\ref{sec:BB}, the exact solution $X$ takes values in $(0,1)$. As in the previous example, we first provide sample paths of the comparison schemes EM, SEM and TE that are all leaving the domain $(0,1)$. See Figure~\ref{num:Nagumo_path_comp}.

\begin{figure}[htp]
\begin{center}
  \includegraphics[width=0.8\textwidth]{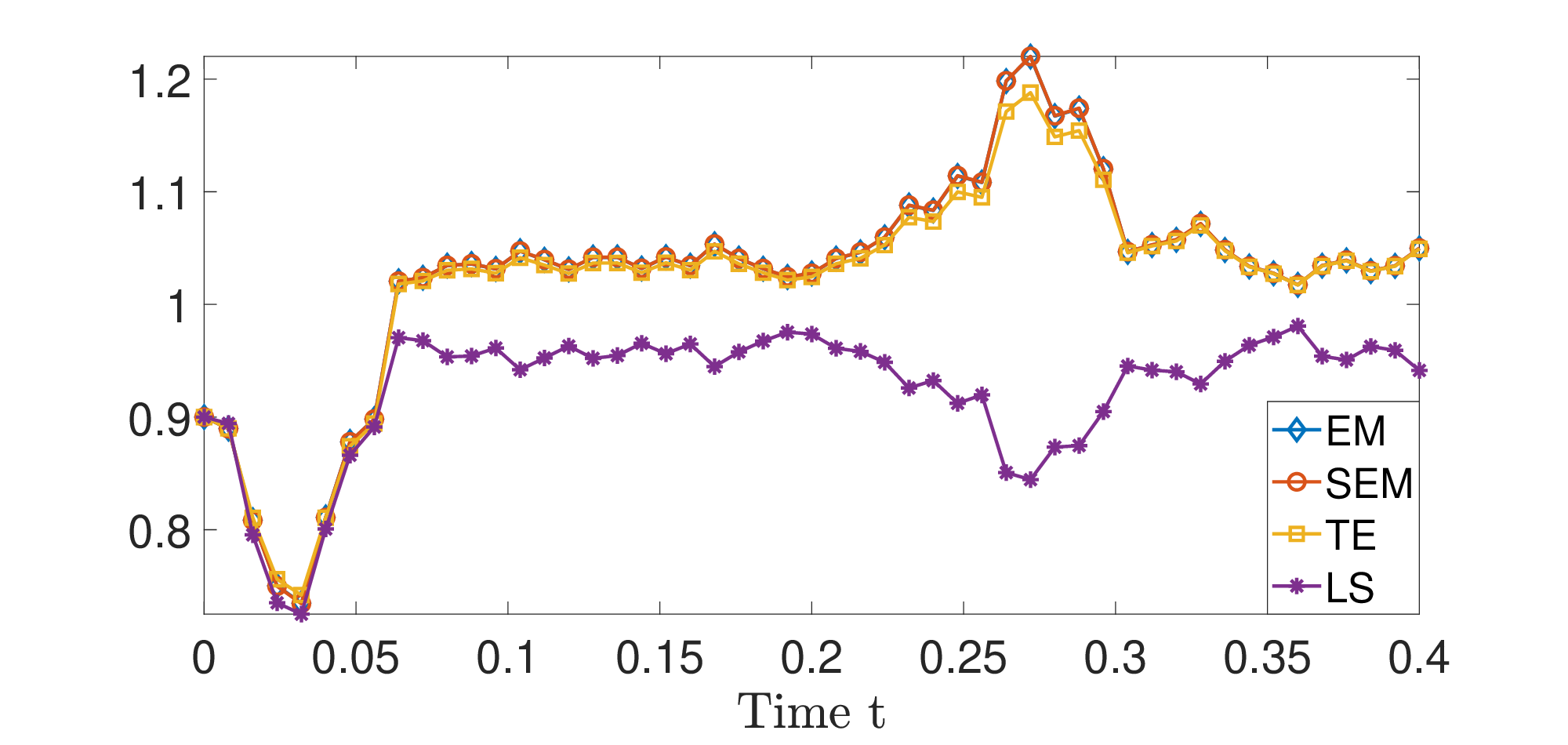}
  \caption{Path comparison of the EM, SEM, TEM and LS schemes applied to the Nagumo SDE with parameters $\lambda = 4$, $x_{0}=0.9$, $T = 0.4$ and $M=50$.}\label{num:Nagumo_path_comp}
  \end{center}
\end{figure}

Similarly to the SIS SDE case, Table~\ref{tb:Nagumo} shows that the integrators EM, SEM, and TE do not preserve the domain $(0,1)$ of the Nagumo SDE and the number of samples that preserve the domain $(0,1)$ decreases as $\lambda>0$ increases. Moreover, Table~\ref{tb:Nagumo} also confirms that the LS scheme preserves the domain $(0,1)$ of the Nagumo SDE. In Table~\ref{tb:Nagumo}, we used $\Delta t=10^{-3}$, $T=1$, $N=100$ number of samples and $x_{0}$ is uniformly distributed on $(0,1)$ for each sample.

\begin{table}[!htbp]
\begin{center}
\begin{tabular}{||c c c c c||} 
 \hline
 $\lambda$ & LS & EM & SEM & TE \\ [0.5ex] 
 \hline\hline
 $6$ & $100/100$ & $100/100$ & $100/100$ & $100/100$\\ 
 \hline
 $7$ & $100/100$ & $95/100$ & $97/100$ & $95/100$ \\ 
 \hline
 $8$ & $100/100$ & $75/100$  & $77/100$& $73/100$ \\ [1ex]
 \hline
\end{tabular}
\caption{Proportion of samples containing only values in $(0,1)$ out of $100$ simulated sample paths for the Lamperti-splitting scheme (LS), the Euler--Maruyama scheme (EM), the semi-implicit Euler--Maruyama scheme (SEM), and the tamed Euler scheme (TE) for the Nagumo SDE for different choices of $\lambda>0$. The parameters used are: $T=1$, $\Delta t = 10^{-3}$ and with $x_{0}$ uniformly distributed on $(0,1)$ for each sample. \label{tb:Nagumo}}
\end{center}
\end{table}
In Figure~\ref{num:Nagumo_conv} we present the $L^{2}(\Omega)$-errors for the same values of $\lambda$ as used in Table~\ref{tb:Nagumo}. The $L^{p}(\Omega)$-error rates in Figure~\ref{num:Nagumo_conv} agree with the rates obtained in Theorem~\ref{theo:main_v3}. We use $T=1$, $N=300$ number of samples to estimate the expected value and $x_{0}$ uniformly distributed on $(0,1)$ for each sample in Figure~\ref{num:Nagumo_conv}.

\begin{figure}[htp]
\begin{center}
  \includegraphics[width=0.8\textwidth]{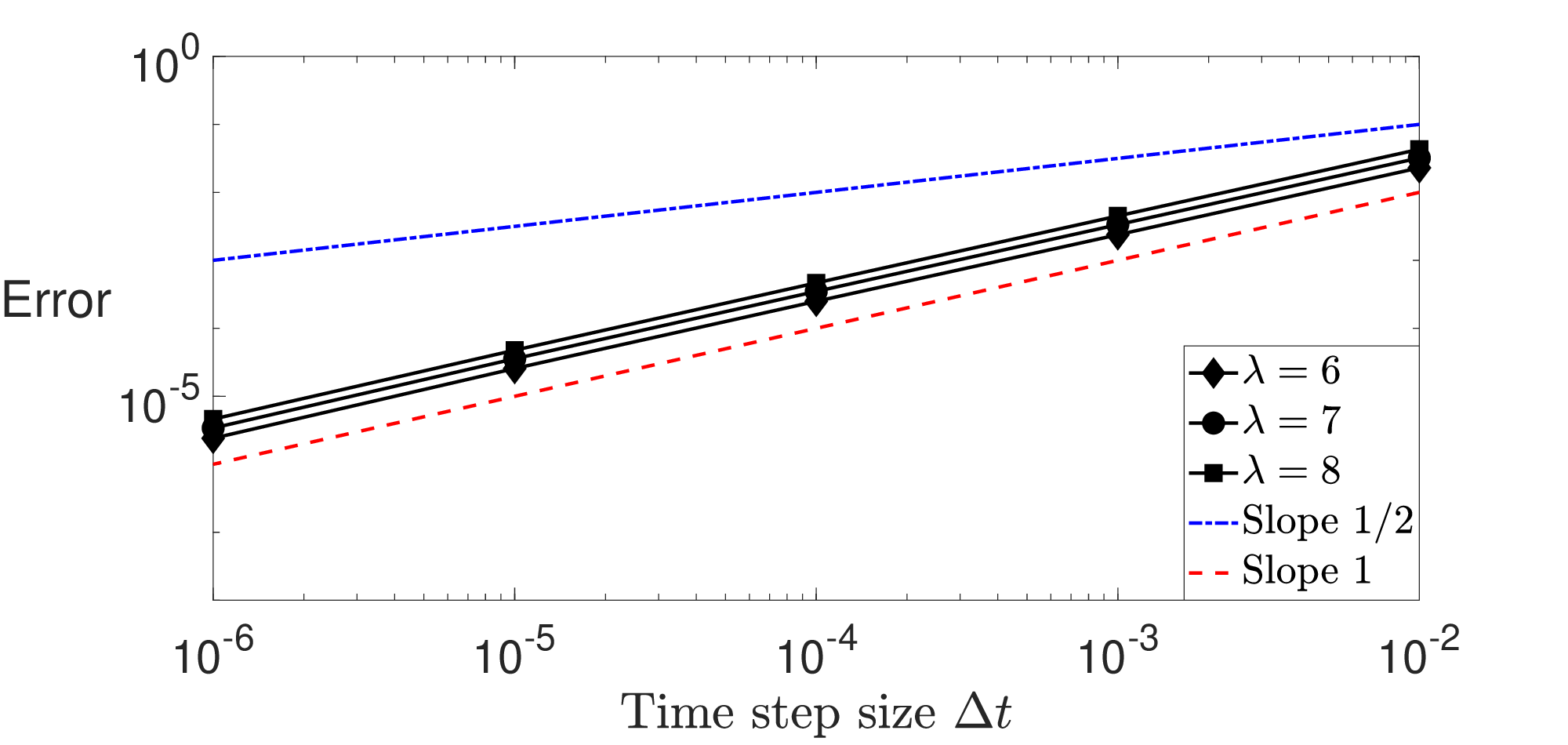}
  \caption{$L^{2}(\Omega)$-errors on the interval $[0,1]$ of the Lamperti-splitting scheme (LS) for the Nagumo SDE for different choices of $\lambda>0$ and reference lines with slopes $1/2$ and $1$. Averaged over $300$ samples.}\label{num:Nagumo_conv}
  \end{center}
\end{figure}

\end{example}
\begin{example}[Allen--Cahn SDE]
Consider the Allen--Cahn type SDE given by
\begin{equation*}
\diff X(t) = \left( X(t) - X(t)^{3} \right) \diff t + \lambda (1-X(t)^2) \diff B(t)
\end{equation*}
with initial value $X(0)=x_{0} \in (-1,1)$;
that is, $f(x) = x - x^{3}$ is cubic and $g(x) = \lambda(1 - x^{2})$ is quadratic in the considered SDE in equation~\eqref{eq:SDEmain}. For the Allen-Cahn type SDE we let $H(x) = (1 + \lambda^2) \Phi^{-1}(x)$ and $\mu = 0$, see Section~\ref{sec:app-AC} for more details about explicit formulas used in the implementation of the LS scheme for the Allen--Cahn SDE. Since the inverse Lamperti transform is given by
\begin{equation*}
\Phi^{-1}(x) = \frac{e^{2 x} - 1}{e^{2 x} + 1},
\end{equation*}
for the case $w_{0} = 0$, see Section~\ref{sec:app-AC} for details, Section~\ref{sec:BB} implies that the exact solution $X$ takes values in $(-1,1)$. We start with providing sample paths where the comparison schemes EM, SEM and TE all leave the domain $(-1,1)$. See Figure~\ref{num:AC_path_comp} below.

\begin{figure}[htp]
\begin{center}
  \includegraphics[width=0.8\textwidth]{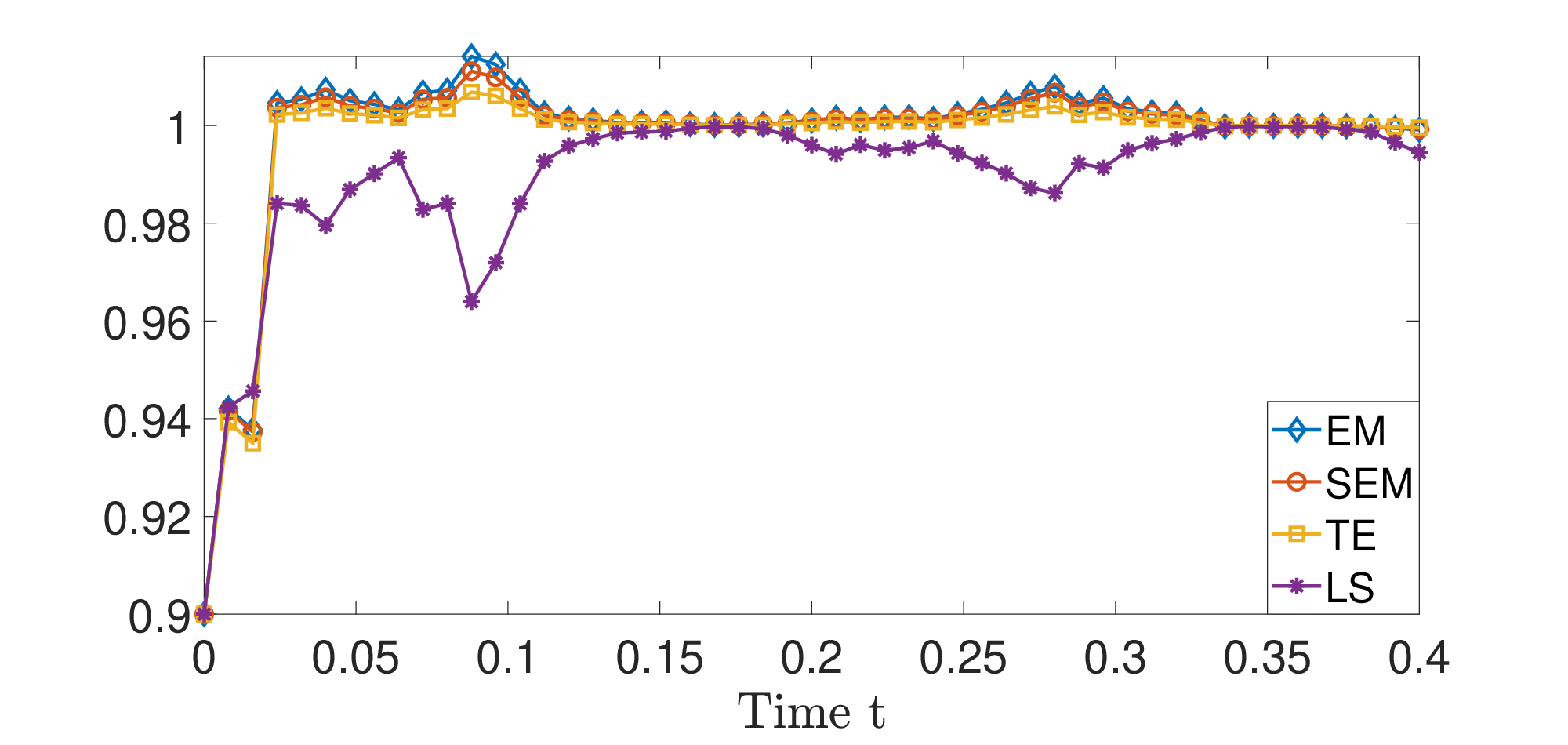}
  \caption{Path comparison of the EM, SEM, TE and LS schemes applied to the Allen-Cahn SDE with parameters $\lambda = 3$, $x_{0}=0.9$, $T = 0.4$ and $M=50$.}\label{num:AC_path_comp}
  \end{center}
\end{figure}

Similarly to the two previous examples, Table~\ref{tb:AC} shows that the integrators EM, SEM and, TE do not preserve the domain $(-1,1)$ of the Allen--Cahn SDE and the number of samples that preserve the domain $(-1,1)$ decreases as $\lambda>0$ increases. Table~\ref{tb:AC} also confirms that the LS schemes preserves the domain $(-1,1)$ of the Allen--Cahn SDE. In Table~\ref{tb:AC}, we used $\Delta t=10^{-3}$, $T=1$, $N=100$ number of samples and $x_{0}$ is uniformly distributed on $(-1,1)$ for each sample.

\begin{table}[!htbp]
\begin{center}
\begin{tabular}{||c c c c c||} 
 \hline
 $\lambda$ & LS & EM & SEM & TE \\ [0.5ex] 
 \hline\hline
 $3$ & $100/100$ & $100/100$ & $100/100$ & $100/100$ \\ 
 \hline
 $3.3$ & $100/100$ & $97/100$ & $97/100$ & $95/100$ \\ 
 \hline
 $3.6$ & $100/100$ & $74/100$ & $89/100$ & $82/100$ \\ [1ex]
 \hline
\end{tabular}
\caption{Proportion of samples containing only values in $(-1,1)$ out of $100$ simulated sample paths for the Lamperti-splitting scheme (LS), the Euler--Maruyama scheme (EM), the semi-implicit Euler--Maruyama scheme (SEM), and the tamed Euler scheme (TE) for the Allen--Cahn type SDE for different choices of $\lambda>0$, $T=1$, $\Delta t = 10^{-3}$ and with $x_{0}$ uniformly distributed on $(-1,1)$ for each sample. \label{tb:AC}}
\end{center}
\end{table}

\begin{figure}[htp]
\begin{center}
  \includegraphics[width=0.8\textwidth]{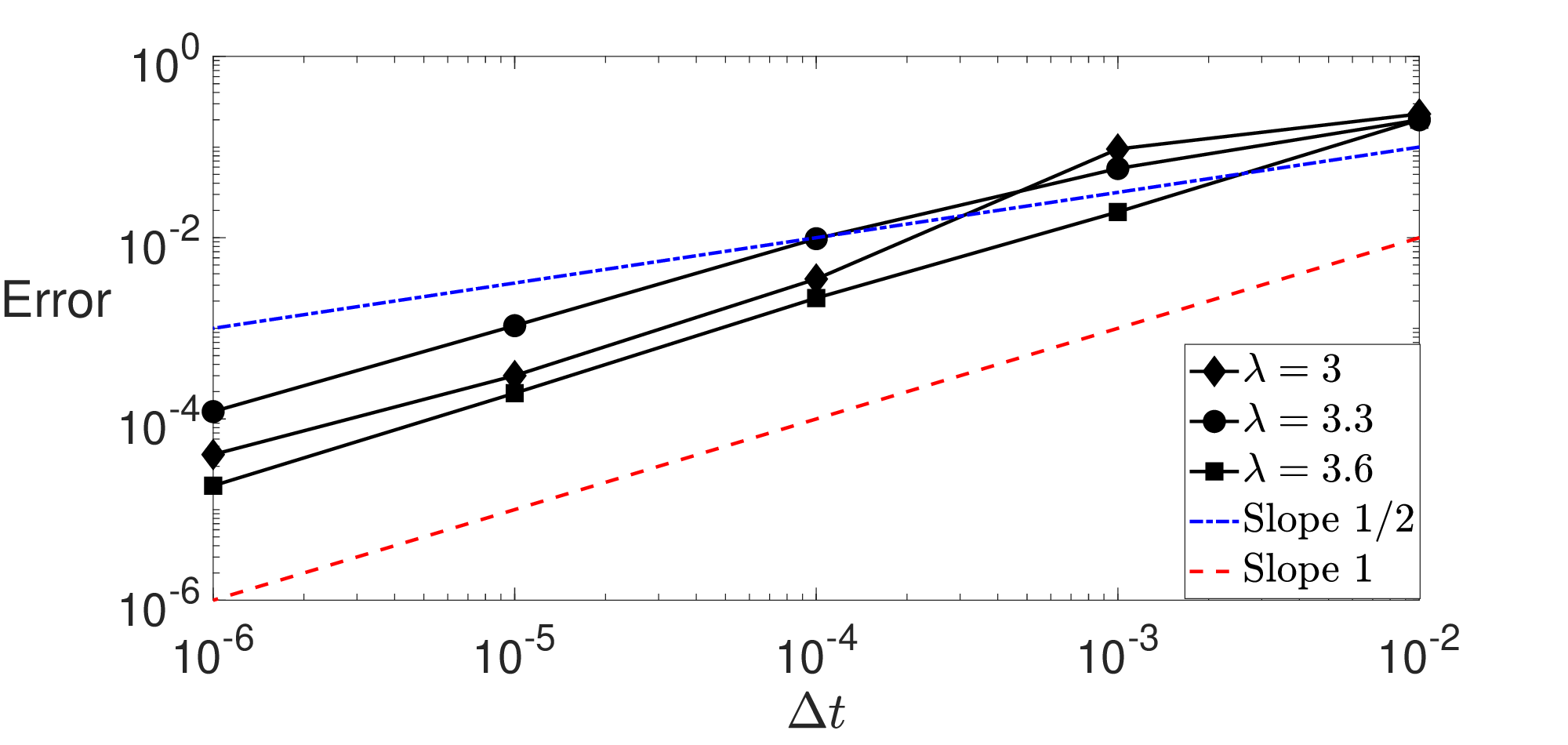}
  \caption{$L^{2}(\Omega)$-errors on the interval $[0,1]$ of the Lamperti-splitting scheme (LS) for the Allen--Cahn type SDE for different choices of $\lambda>0$ and reference lines with slopes $1/2$ and $1$. Averaged over $300$ samples.}\label{num:AC_conv}
  \end{center}
\end{figure}

In Figure~\ref{num:AC_conv} we present the $L^{2}(\Omega)$-errors for the same values of $\lambda$ as used in Table~\ref{tb:AC}. The $L^{2}(\Omega)$-error rates in Figure~\ref{num:AC_conv} agree with the rates obtained in Theorem~\ref{theo:main_v2}. We use $T=1$, $N=300$ number of samples to estimate the expected value and $x_{0}$ uniformly distributed on $(0,1)$ for each sample in Figure~\ref{num:AC_conv}.
\end{example}

\begin{appendix}
\section{Proof of a lemma}\label{sec:appA}
\begin{lemma}\label{lem:aux}
Recall that $\ell(s) = t_{k}$ whenever $s \in [t_{k},t_{k+1})$ and that $\left( B(t) \right)_{t \in [0,T]}$ denotes a standard Brownian motion. Then, for every $p \geq 1$, it holds that
\begin{equation*}
\E \left[ \left| \int_{0}^{T} \int_{\ell(s)}^{s} H'(Y(r)) \diff B(r) \diff s \right|^{p} \right] \leq C(p) L_{H} T^{p/2} \Delta t^{p}
\end{equation*}
where $C(p)$ is the BDG constant.
\end{lemma}
\begin{proof}[Proof of Lemma~\ref{lem:aux}]
The proof essentially consists of applying the stochastic Fubini theorem followed by the BDG inequality. We first re-write the integral as follows
\begin{equation*}
\int_{0}^{T} \int_{\ell(s)}^{s} H'(Y(r)) \diff B(r) \diff s = \int_{0}^{T} \int_{0}^{T} \mathbbm{1}_{\{ r \in (\ell(s),s) \}}(r,s) H'(Y(r)) \diff B(r) \diff s
\end{equation*}
where $\mathbbm{1}_{A}(x)=1$ if $x \in A$ and $\mathbbm{1}_{A}(x)=0$ if $x \not\in A$. We now apply the stochastic Fubini's theorem
\begin{equation*}
\E \left[ \left| \int_{0}^{T} \int_{\ell(s)}^{s} H'(Y(r)) \diff B(r) \diff s \right|^{p} \right] = \E \left[ \left| \int_{0}^{T} \int_{0}^{T} \mathbbm{1}_{\{ r \in (\ell(s),s) \}}(r,s) H'(Y(r)) \diff s \diff B(r) \right|^{p} \right]
\end{equation*}
followed by the Burkholder-Davis-Gundy (BDG) inequality to estimate
\begin{multline}
\E \left[ \left| \int_{0}^{T} \int_{0}^{T} \mathbbm{1}_{\{ r \in (\ell(s),s) \}}(r,s) H'(Y(r)) \diff s \diff B(r) \right|^{p} \right] \nonumber \\ \leq C(p) \E \left[ \left( \int_{0}^{T} \left( \int_{0}^{T} \mathbbm{1}_{\{ r \in (\ell(s),s) \}}(r,s) H'(Y(r)) \diff s \right)^{2} \diff r \right)^{p/2} \right],
\end{multline}
where $C(p)$ is the constant from the BDG inequality. We now use that the inner integral can be expressed as
\begin{equation*}
\begin{split}
\int_{0}^{T} \mathbbm{1}_{\{ r \in (\ell(s),s) \}}(r,s) H'(Y(r)) \diff s &= H'(Y(r)) \int_{0}^{T} \mathbbm{1}_{\{ r \in (\ell(s),s) \}}(r,s) \diff s \\ &= H'(Y(r)) \left( \ell(r+\Delta t) - r \right)
\end{split}
\end{equation*}
and so we can estimate
\begin{equation*}
\begin{split}
\left( \int_{0}^{T} \mathbbm{1}_{\{ r \in (\ell(s),s) \}}(r,s) H'(Y(r)) \diff s \right)^{2} &= |H'(Y(r))|^{2} \left( \ell(r+\Delta t) - r \right)^{2} \\ &\leq |H'(Y(r))|^{2} \Delta t^{2}.
\end{split}
\end{equation*}
Collecting everything gives us the desired estimate
\begin{equation*}
\begin{split}
\E \left[ \left| \int_{0}^{T} \int_{\ell(s)}^{s} H'(Y(r)) \diff B(r) \diff s \right|^{p} \right] &\leq C(p) \E \left[ \left( \int_{0}^{T} |H'(Y(r))|^{2} \Delta t^{2} \diff r \right)^{p/2} \right] \\ &= C(p) \sup_{x \in \mathbb{R}} |H'(x)| T^{p/2} \Delta t^{p} \\ &\leq C(p) L_{H} T^{p/2} \Delta t^{p},
\end{split}
\end{equation*}
where we in the last inequality used Proposition~\ref{prop:H}.
\end{proof}

\section{Additional formulas}\label{sec:app}
Here we provide a detailed description of the LS scheme for the three considered examples in Section~\ref{sec:num}. We present explicit formulas for both $y(t)$ in equation~\eqref{eq:ODE-LampMain} and for $X^{LS} = \Phi^{-1}(Y^{LS})$. We denote by $\log$ the natural logarithm.

\subsection{SIS SDE}\label{sec:app-SIS}
Recall that the SIS epidemic model is given by
\begin{equation*}
dX(t) = X(t) \left(1- X(t) \right) dt + \lambda X(t)(1-X(t))dB(t)
\end{equation*}
with initial value $X(0)=x_{0} \in (0,1)$. The boundary points $\{0,1 \}$ are stationary points: if $x_{0} \in \{ 0,1 \}$, then $X(t) = x_{0}$ for all times $t>0$. Let $w_{0} \in (0,1)$. Direct computations give
\begin{align*}
  \frac{f(x)}{g(x)} - \frac{\lambda^2}{2} g'(x) & = \lambda^2 x + (1 - \lambda^2/2), \\
  \Phi(x) & =\log(x) - \log(1-x) - \log(w_{0}) + \log(1-w_{0})
\end{align*}
and
\begin{equation*}
\Phi^{-1}(x) = \frac{w_{0} e^{x}}{w_{0} e^{x} + (1-w_{0})}.
\end{equation*}
If we let $H(x) = \lambda^2 \Phi^{-1}(x)$ and $\mu = 1 - \lambda^2/2$, then the assumptions in Section~\ref{sec:setting} and Section~\ref{sec:scheme_v2} are fulfilled: Assumptions~\ref{ass:f} and~\ref{ass:g},~\ref{ass:fg} are easily checked and the choice of $H$ implies that the ODE 
\begin{equation}\label{eq:SISODE}
\left\lbrace
\begin{aligned}
& \frac{\diff y(t)}{\diff t}= H(y(t)) = \lambda^2 \Phi^{-1}(y(t)),\\ 
& y(t_{m}) = \Phi(x_{m}),
\end{aligned}
\right.
\end{equation}
for $t \in [t_{m},t_{m+1}]$, where $x_{m} = X^{LS}(t_{m})$, has the solution formula given by
\begin{equation}\label{eq:SISy}
\begin{split}
y(t) &= W \left( \frac{(1-x_{m})e^{\frac{1-x_{m}}{x_{m}}}}{x_{m} e^{\lambda^{2} (t-t_{m})}} \right) + \lambda^2 (t-t_{m}) \\ &+ \log \left( \frac{x_{m}}{1-x_{m}} \frac{1-w_{0}}{w_{0}} \right) - \frac{1-x_{m}}{x_{m}},
\end{split}
\end{equation}
where $W$ is the Lambert W function \cite{LampW}. Since the above formula is not implementable in practice, we use the standard Euler scheme to approximate the solution $y$ in~\eqref{eq:SISy} to the ODE in~\eqref{eq:SISODE}. By inserting the explicit formula for $y(t)$ in equation~\eqref{eq:SISy} into the defining formula for $Y^{LS}(t)$ in equation~\eqref{eq:LampSplitDisc}, we obtain, after simplifications, that
\begin{equation*}
\begin{split}
X^{LS}(t_{m+1}) &= \Phi^{-1}(Y^{LS}(t_{m+1})) \\ &= \frac{e^{(1 - (\lambda^{2})/2) (t_{m+1}-t_{m})} e^{\lambda (B(t_{m+1}) - B(t_{m}))}}{e^{(1 - (\lambda^{2})/2) (t_{m+1}-t_{m})} e^{\lambda (B(t_{m+1}) - B(t_{m}))} + W \left( \frac{(1-x_{m})e^{\frac{1-x_{m}}{x_{m}}}}{x_{m} e^{\lambda^{2} (t_{m+1}-t_{m})}} \right)}.
\end{split}
\end{equation*}

\subsection{Nagumo SDE}\label{sec:app-NAG}
Recall that the Nagumo SDE is given by
\begin{equation*}
\diff X(t) = - X(t)(1-X(t))(1-X(t)) \diff t - \lambda X(t) (1-X(t)) \diff B(t)
\end{equation*}
with initial value $X(0) = x_{0} \in (0,1)$. The boundary points $\{ 0,1 \}$ are stationary points: if $x_{0} \in \{ 0,1 \}$, then $X(t)=x_{0}$ for all times $t>0$. Let $w_{0} \in (0,1)$. Similarly to the SIS SDE in Section~\ref{sec:app-SIS}, direct computations give us
\begin{align*}
  \frac{f(x)}{g(x)} - \frac{\lambda^2}{2} g'(x) & = \left( 1 + \frac{\lambda^2}{2} \right) - \left(1 + \lambda^2 \right)x, \\
  \Phi(x) & = \log(1-x) - \log(x) - \log(1-w_{0}) + \log(w_{0})
\end{align*}
and
\begin{equation*}
\Phi^{-1}(x) = \frac{w_{0}}{(1-w_{0})e^{ x} + w_{0}}.
\end{equation*}
Let now $H(x) = -(1 + \lambda^2) \Phi^{-1}(x)$ and $\mu = (1 + \frac{\lambda^2}{2})$. One checks that the assumptions in Section~\ref{sec:setting} and Section~\ref{sec:scheme_v2} are fulfilled: Assumptions~\ref{ass:f},~\ref{ass:g} and~\ref{ass:fg} are easily verified and the ODE
\begin{equation}\label{eq:NanODE}
\left\lbrace
\begin{aligned}
& \frac{\diff y(t)}{\diff t}= -(1 + \lambda^2) \Phi^{-1}(y(t)),\\ 
& y(t_{m}) = \Phi(x_{m}),
\end{aligned}
\right.
\end{equation}
for $t \in [t_{m},t_{m+1}]$, where $x_{m} = X^{LS}(t_{m})$, has the solution formula given by
\begin{equation}\label{eq:Nany}
\begin{split}
y(t) &= - W \left( \frac{(1-x_{m})e^{\frac{1-x_{m}}{x_{m}}}}{x_{m} e^{(1+\lambda^{2}) (t-t_{m})}} \right) - (1 + \lambda^2) (t-t_{m}) \\ &+ \log \left(\frac{1-x_{m}}{x_{m}} \frac{w_{0}}{1-w_{0}} \right) + \frac{1-x_{m}}{ x_{m}},
\end{split}
\end{equation}
where $W$ is the Lambert W function \cite{LampW}. Since the above formula is not implementable in practice, we use the standard Euler scheme to approximate the solution $y$ in~\eqref{eq:Nany} to the ODE in~\eqref{eq:NanODE}. We insert the formula in equation~\eqref{eq:Nany} into equation~\eqref{eq:LampSplitDisc} to obtain, after simplifications, that
\begin{multline*}
X^{LS}(t_{m+1}) = \Phi^{-1}(Y^{LS}(t_{m+1})) \\ = \left( W \left( \frac{(1-x_{m})e^{\frac{1-x_{m}}{x_{m}}}}{x_{m} e^{(1+\lambda^{2}) (t_{m+1}-t_{m})}} \right) e^{(1 + (\lambda^{2})/2) (t_{m+1}-t_{m}) + \lambda (B(t_{m+1}) - B(t_{m}))} + 1 \right)^{-1}.
\end{multline*}

\subsection{Allen--Cahn SDE}\label{sec:app-AC}
Recall that the Allen--Cahn type SDE is given by
\begin{equation*}
\diff X(t) = \left( X(t) - X(t)^{3} \right) \diff t + \lambda (1-X(t)^2) \diff B(t)
\end{equation*}
with initial value $X(0)=x_{0} \in (-1,1)$.
The boundary points $\{ -1,1 \}$ are stationary points: if $x_{0} \in \{ -1,1 \}$, then $X(t)=x_{0}$ for all times $t>0$. Observe that $0$ is not stationary, since $g(0) \neq 0$. In this case, we present the implementation formulas for the choice $w_{0} = 0$ as this simplifies the expressions. Straightforward computations give us
\begin{align*}
  \frac{f(x)}{g(x)} - \frac{\lambda^2}{2} g'(x) & = (1 + \lambda^2) x, \\
  \Phi(x) & = \frac{1}{2} \left( \log(1+x) - \log(1-x) \right)
\end{align*}
and
\begin{equation*}
\Phi^{-1}(x) = \frac{e^{2 x} - 1}{e^{2 x} + 1}.
\end{equation*}
We let $H(x) = (1 + \lambda^2) \Phi^{-1}(x)$ and $\mu = 0$. Then assumptions in Section~\ref{sec:setting} and Section~\ref{sec:scheme} are fulfilled: Assumptions~\ref{ass:f},~~ \ref{ass:g} and~\ref{ass:fg} are easily verified and the ODE
\begin{equation*}
\left\lbrace
\begin{aligned}
& \frac{\diff y(t)}{\diff t}= (1 + \lambda^2) \Phi^{-1}(y(t)),\\ 
& y(t_{m}) = y_{m}=\ \Phi(x_{m}),
\end{aligned}
\right.
\end{equation*}
for $t \in [t_{m},t_{m+1}]$, where $x_{m} = X^{LS}(t_{m})$, has the explicit solution formula given by

\begin{equation}\label{eq:ACy}
  \begin{split}
    y(t)
    & = \log \left( \frac{1}{2} \left( \sqrt{ e^{2 (1+\lambda^2) (t-t_{m})} \left( e^{-y_{m}} - e^{y_{m}} \right)^{2} + 4} \right. \right.\\
    & \quad \left. \left.- e^{(1+\lambda^2)(t-t_{m})} \left( e^{-y_{m}} - e^{y_{m}} \right) \right) \right).
  \end{split}
\end{equation}
Combining equation~\eqref{eq:ACy} with equation~\eqref{eq:LampSplitDisc} gives us
\begin{equation*}
\begin{split}
X^{LS}(t_{m+1}) &= \Phi^{-1}(Y^{LS}(t_{m+1})) \\ &= \frac{V(t_{m+1}) e^{2 \lambda \left( B(t_{m+1}) - B(t_{m}) \right)} -(1-x_{m})(1+x_{m})}{V(t_{m+1}) e^{2 \lambda \left( B(t_{m+1}) - B(t_{m}) \right)} + (1-x_{m})(1+x_{m})},
\end{split}
\end{equation*}
where $x_{m} = X^{LS}(t_{m})$ and where
\begin{equation*}
V(t) = \left( \sqrt{ \left(x_{m} \right)^2 e^{2 (1+\lambda^2) (t-t_{m})} + (1-x_{m}) (1+x_{m})} + x_{m} e^{(1+\lambda^2) (t-t_{m})} \right)^2.
\end{equation*} 

\end{appendix}

\section*{Acknowledgements}
The author would like to thank the anonymous referees for their comments and suggestions which helped to improve the quality of this article. The author would also like to give a special thank you to David Cohen for his comments, suggestions and sharing some of his codes. Lastly, the author would also like to thank Charles-Edouard Bréhier for reading and discussing the original and revised versions of the paper. This work is partially supported by the Swedish Research Council (VR) (projects nr. $2018-04443$). The computations were enabled by resources provided by the National Academic Infrastructure for Supercomputing in Sweden (NAISS) and the Swedish National Infrastructure for Computing (SNIC) at UPPMAX, Uppsala University, partially funded by the Swedish Research Council through grant agreements no. 2022-06725 and no. 2018-05973.

\bibliographystyle{abbrv}
\bibliography{LSbib}

\end{sloppypar}
\end{document}